\documentclass[11pt,oneside]{article}
\usepackage[square, numbers]{natbib}
\usepackage{graphicx,setspace,import, tikz}
\graphicspath{ {./finalPlots/} }
\usepackage{tocloft}
\usepackage{amsmath, amssymb}
\usepackage{amsthm,xcolor}
\usepackage{mathtools}
\usepackage{amsfonts}
\usepackage{mathrsfs}
\usepackage{hyperref}
\usepackage{enumerate}
\usepackage{bm, bbm, cancel}
\usepackage{hhline}

\providecommand{\keywords}[1]
{
  \small	
  \textbf{\textit{Keywords---}} #1
}

\newcommand{\R}{\mathbb{R}}
\newcommand{\N}{\mathbb{N}}

\newcommand{\T}{\mathbb{T}}

\def\E{\mathbb E}

\def\P{\mathbb P}

\numberwithin{equation}{section}

\newtheorem{theorem}{Theorem}[section] 
\newtheorem{lemma}[theorem]{Lemma}
\newtheorem{corollary}[theorem]{Corollary}

\newtheorem{proposition}[theorem]{Proposition}

\theoremstyle{definition}
\newtheorem{definition}[theorem]{Definition}
\newtheorem{assumption}{Assumption}

\newtheorem{remark}[theorem]{Remark}

\newtheorem{examplex}{Example}
\newenvironment{example}
{\pushQED{\qed}\examplex}
{\popQED\endexamplex}

\title{Banach spaces of continuous paths with finite $p$-th variation}

\vspace{2mm}

\author{
	\textsc{Purba Das}
	\thanks{Department of Mathematics, King's College London, UK (E-mail: \it{purba.das@kcl.ac.uk})}
	\and
	\textsc{Donghan Kim} 
	\thanks{Department of Mathematical Sciences, KAIST, South Korea (E-mail: {\it kimdonghan@kaist.ac.kr})}
    \and
    \textsc{Fang Rui Lim}
    \thanks{Department of Mathematics, University of Michigan, US (E-mail:\it{ruilim@umich.edu})}
}
\date{}

\begin{document}
    \maketitle
	
\begin{abstract}
    \noindent We study pathwise $p$-th variation of continuous paths on a compact interval along a fixed partition sequence. Although the class of continuous paths with finite $p$-th variation is generally not linear, we develop a coefficient-based approach via Faber-Schauder expansions that, for any $p>1$, enables the construction of paths with prescribed $p$-th variation while preserving useful linear structures and H\"older regularity. We first construct continuous paths with linear $p$-th variation from suitable conditions on their Faber-Schauder coefficients. We then prescribe nonlinear $p$-th variation through a multiplicative transformation and show that, whenever nonempty, the class of H\"older continuous paths with a given $p$-th variation is dense in $C([0,1])$. Next, we introduce a transport procedure that turns a Banach subspace of continuous functions into a Banach subspace of paths with explicitly controlled $p$-th variation. We also prove stability of the associated pathwise F\"ollmer-It\^o map on these transported subspaces. Finally, via time-changes, we show that this constructive framework extends from $q$-adic partition sequences to broader classes of dense $q$-refining partition sequences.
\end{abstract}
	
    \smallskip
    
    
    \tableofcontents

\section{Introduction}

Since the seminal work of \citet{follmer1981}, pathwise It\^o calculus, a deterministic analogue of the stochastic integration theory of \citet{ito1941}, has provided a powerful framework for defining stochastic integrals and change-of-variable formulas without relying on any underlying probabilistic structure. In this approach, the pathwise quadratic variation, or more generally the pathwise $p$-th variation, of a path along a fixed partition sequence plays a central role in the existence of the pathwise integral.  The $p$-th variation of a continuous path thus provides a deterministic way to quantify the oscillatory behavior of a path without probabilistic assumptions. Subsequent developments have significantly extended the scope of pathwise It\^o-type formulas and clarified their connections to rough path theory and functional It\^o calculus \cite{chiu:causal, ContFournie2013, cont2010, ruhong2022, perkowski2019}, \cite[Chapter 6]{Dupire2026}, \cite[Chapter 5]{FrizHairer}. These developments have also found important applications in mathematical finance \cite{Valkeila2008, Chiu:model-free, Follmer:book, Karatzas:Kim, schied2014, Schied:model-free, Vovk}. 

Despite these advances, a fundamental structural difficulty remains. As observed by \citet{schied2016}, natural classes of paths that act as integrators in F\"ollmer's pathwise integral, such as spaces of paths with finite $p$-th variation along a fixed partition sequence, are typically \textit{not} linear spaces. In particular, the dyadic $p$-th variation space $V^p_{\mathbb{T}}$ (see Definition~\ref{Def: p-th variation}) fails in general to be closed under addition, which poses a fundamental obstacle to approximation arguments, the development of stable integration theories and the formulation of a robust functional-analytic framework for rough paths.

This lack of linear structure raises a natural question: to what extent can one recover useful analytical structure—such as linearity or stability—while retaining precise control of the $p$-th variation along a given partition sequence? This issue is central both for the development of pathwise stochastic calculus, where one seeks sufficiently rich and stable classes of paths with controlled variation in order to define integrals and study their continuity properties, and for approximation problems, where one aims to construct dense families of paths with prescribed roughness.

In a recent paper \cite{das-kim2024}, the generalized $p$-th variation class $\mathcal{X}^p_{\mathbb{T}}$ (see \eqref{def : x p pi space}) was introduced in the course of a systematic study of pathwise roughness for continuous functions, defined through asymptotic boundedness of $p$-th variation along a fixed partition sequence. This space arises naturally when one seeks to characterize roughness properties of continuous paths without assuming the existence of a limiting $p$-th variation, and it admits a convenient Banach space structure. Its focus, however, was on asymptotic control and roughness characterization rather than on the explicit construction of paths with prescribed $p$-th variation or on the internal structure of the class $V^p_{\mathbb{T}}$ itself. By contrast, constructing paths with a prescribed finite $p$-th variation is more delicate, since one must ensure not only asymptotic boundedness but also the existence and precise identification of the limiting variation.

From a modeling and approximation-theoretic perspective, systematic examples of continuous paths with prescribed $p$-th variation remain remarkably scarce, especially beyond the Gaussian or semimartingale setting. Classically, irregular paths are often produced through probabilistic mechanisms, most notably via Volterra-type representations of stochastic processes, which yield trajectories with low regularity, such as fractional Brownian motion with Hurst parameter $H<\frac12$. While such representations have greatly expanded the class of models accessible to stochastic analysis, they rely inherently on probabilistic structure and do not provide a pathwise procedure for transforming a given realization into one with a prescribed variation.

Consequently, many results in pathwise calculus have been formulated abstractly for classes of paths satisfying suitable variation conditions, while comparatively little is known about how rich these classes are or how flexible they are from a constructive viewpoint. To the best of our knowledge, in a non-Gaussian setup, apart from the example introduced by \citet{Misura2019,schied2020fractal}, which yield paths whose $p$-th variation exhibits linear growth, there are no explicit families of continuous paths with nontrivial and controllable $p$-th variation. This lack of concrete examples poses an obstacle to the further development of pathwise calculus, since it remains unclear whether variation-based conditions severely restrict admissible paths or instead allow for a broad range of irregular continuous functions. Addressing this issue requires characterizations of $p$-th variation, as well as systematic construction of such paths and proving that they are abundant from the viewpoint of approximation.

In this manuscript, we develop a constructive framework for producing continuous paths with prescribed $p$-th variation and for identifying linear structures inside the nonlinear $p$-variation space $V^p_{\mathbb{T}}$. Our approach is based on two complementary ideas. The first is to construct explicit reference paths whose $p$-th variation grows linearly in time. The second is to use these reference paths as multiplicative transports: by multiplying them with paths of vanishing $p$-th variation, we generate broad families of continuous functions with explicitly computable $p$-th variation.

This yields a simple yet powerful deterministic procedure for realizing a large class of target variation functions and shows that nontrivial prescribed $p$-th variation can be generated in a robust and flexible manner.

Beyond the construction, we show that any nonempty class of continuous functions with a given prescribed $p$-th variation is dense in $C([0,1])$. Thus, paths with prescribed $p$-th variation form large subclasses of the space of continuous functions. We then use the same construction principle to identify concrete linear structures inside the nonlinear space $V^p_{\mathbb{T}}$. Starting from Banach spaces of continuous functions with vanishing $p$-th variation, we construct transported Banach subspaces of $V^p_{\mathbb{T}}$ by multiplication with a fixed positive reference path of linear $p$-th variation. On these subspaces, the $p$-th variation admits an explicit formula and varies continuously with respect to the transported norm. When $p$ is an even integer, we further prove continuity of the corresponding pathwise F\"ollmer-It\^o map. This provides a functional-analytic setting in which both the variation and the induced pathwise calculus are stable under perturbation.

Although the main results of the paper are formulated first for the dyadic partition sequence, this choice is made primarily for simplicity of exposition. In Section \ref{sec: extention to general partitions}, we show that the constructive framework developed here is not confined to the dyadic setting.

In this way, the multiplicative construction of prescribed $p$-th variation, together with the accompanying density, Banach subspace, and stability results, extends beyond the dyadic framework to a more flexible family of partition sequences. 

\medskip

\noindent\textit{Preview:} This paper is organized as follows. Section~\ref{sec: preliminary} introduces some notations and preliminary results and Section~\ref{sec: results} provides our main results along a dyadic partition sequence. Section~\ref{sec: extention to general partitions} extends the results of Section~\ref{sec: results} to a more general class of partition sequences. Finally, Section \ref{sec: proofs} includes some lengthy proofs.

\bigskip

\section{Faber-Schauder expansions} \label{sec: preliminary}

In this section, we introduce several notations that will be used throughout the paper. We then explain the Faber-Schauder representation of continuous functions, which will be the main tool in developing our results, and state preliminary results on the Faber-Schauder coefficients. 

\subsection{Notations}
Throughout this paper, we shall work with continuous paths in $C([0, 1])$, a space of real-valued continuous functions on $[0, 1]$. Even though we restrict ourselves to the unit interval $[0, 1]$ for simplicity, our results in the paper can be generalized to continuous functions defined on $[0, T]$ for any fixed $T > 0$. For $\alpha \in (0, 1)$, we denote $C^{\alpha}([0, 1])$ the subspace of $\alpha$-H\"older continuous functions, i.e.,
\begin{equation*}
    C^{\alpha}([0, 1]) := \bigg\{ x\in C([0,1]) ~ \bigg| ~ \sup_{\substack{t,s \in [0,1] \\ t\neq s }}\frac{\vert x(t)-x(s) \vert}{|t-s|^{\alpha}} < \infty  \bigg\},
\end{equation*}
and the $\alpha$-H\"older norm of $x \in C^{\alpha}([0, 1])$
\begin{equation}    \label{def.Holder norm}
	\Vert x \Vert_{C^{\alpha}} := |x(0)|+\sup_{\substack{t,s \in [0,1] \\ t\neq s }}\frac{\vert x(t)-x(s) \vert}{|t-s|^{\alpha}}.
\end{equation}
For $p \in [1, \infty]$, we denote by $L^p([0, 1])$ the usual $L^p$-space of measurable functions $x$ on $[0, 1]$ satisfying
\[
    \|x\|_{L^p} := \bigg( \int_0^1 |x(s)|^p ds \bigg)^{\frac1p} < \infty, \qquad \text{or} \qquad \|x\|_\infty := \sup_{s\in[0,1]} |x(s)| < \infty.
\]

On the interval $[0, 1]$, we consider the dyadic partition sequence $\T = (\T^m)_{m \ge 0}$ which consists of the dyadic points $t^m_k := k2^{-m}$
\begin{equation}    \label{def: dyadic partition}
    \T^m := \Big\{ 0, \frac{1}{2^m}, \frac{2}{2^m}, \cdots, 1 \Big\}, \qquad m \ge 0.
\end{equation}
For each nonnegative integer $m \ge 0$, let us denote a set of integers $I_m := \{0, 1, \cdots, 2^m -1\}$.

\begin{definition}  \label{Def: p-th variation}
    For any $x \in C([0, 1])$ and $p \ge 1$, we denote 
    \begin{equation}    \label{def: discrete p-th variation}
    	[x]^{(p)}_{\T^n}(t) := \sum_{j \in I_n} \big \vert x(t^n_{j + 1} \wedge t) - x(t^n_j \wedge t)\big\vert^p, \qquad \forall \, t \in [0, 1],
    \end{equation}
    the \textit{$p$-th variation of $x$ along the $n$-th dyadic partition $\T^n$} for each $n \ge 0$. If the limit of $[x]_{\T^n}^{(p)}(t)$ as $n \to \infty$ exists and is continuous, then we say that \textit{$x$ admits finite $p$-th variation along the dyadic partition sequence $\T$}, and write its limit as
    \begin{equation}	\label{eq : uniform limit of p-th variation}
        [x]^{(p)}_{\T}(t) := \lim_{n \to \infty} [x]_{\T^n}^{(p)}(t), \qquad \forall \, t \in [0, 1].
    \end{equation}
    In such case, the convergence is uniform in $t$ and the limit $t \mapsto [x]^{(p)}_{\T}(t)$ is non-decreasing \cite[Definition 1.1 and Lemma 1.3]{perkowski2019}. We denote by $V^{p}_{\T}$ the space of such functions $x$ admitting finite $p$-th variation along $\T$.
\end{definition}

We note that $V^2_{\T}$ corresponds to the space of (continuous) functions to which we can apply F\"ollmer's pathwise It\^o formula \cite{follmer1981}; recently, \citet{perkowski2019} extended the formula to the class $V^p_{\T}$ for any even integer $p \in 2\N$. On the other hand, Schied \cite[Proposition 2.7]{schied2016} found an example of a pair $x, y \in V^2_{\T}$ but $x+y \notin V^2_{\T}$. This suggests that $V^2_{\T}$ is not a vector (linear) space.

We may consider a strictly larger subclass of $C([0, 1])$ than $V^p_{\T}$; for any $x \in C([0, 1])$ and $p \ge 1$, we define
\begin{equation*}
	\Vert x \Vert^{(p)}_{\T} := |x(0)| + \sup_{n \ge 0} \, \Big([x]_{\T^n}^{(p)} (1)\Big)^{\frac{1}{p}},
\end{equation*}
and consider the subspace of $C([0, 1])$:
\begin{equation}    \label{def : x p pi space}
	\mathcal X^{p}_{\T} := \big\{x \in C([0, 1])\, :  \Vert x \Vert^{(p)}_{\T} < \infty \big\}.
\end{equation}
A recent paper \cite[Proposition 2.5]{das-kim2024} showed that $\Vert \cdot \Vert^{(p)}_{\T}$ is a norm and that $\mathcal{X}^p_{\T}$ is a Banach space under this norm. It is straightforward to check the inclusion $V^p_{\T} \subset \mathcal{X}^p_{\T}$ for any $p \ge 1$, since the existence of the limit $[x]^{(p)}_{\T}(1)$ implies $\sup_{n\ge0}[x]^{(p)}_{\T^n}(1)<\infty$.

The Banach space $\mathcal{X}^p_{\T}$ can be viewed as a linear enlargement of $V^p_{\T}$, tailored to quantify the roughness of a continuous path via its variation index, i.e., the smallest $p$ for which the discrete $p$-th variations $[x]^{(p)}_{\T^n}(t)$ in \eqref{def: discrete p-th variation} remain uniformly bounded in $n$. However, F\"ollmer-type pathwise It\^o formulas do not apply to arbitrary elements of $\mathcal{X}^p_{\T}$. Motivated by this, in Section~\ref{subsec: Banach subspaces} we construct explicit linear subspaces of $V^p_{\T}$ consisting of paths with controlled dyadic $p$-th variation.

We conclude this subsection with the following lemma, which will be used in the subsequent sections.

\begin{lemma}   \label{lem:inclusions}
    For $1\le p<q$, the following statements hold.
    \begin{enumerate}[(i)]
        \item If $x\in V^p_{\T}$, then $[x]^{(q)}_{\T}\equiv0$, and in particular
        $V^p_{\T}\subset V^q_{\T}$.
        \item We have the continuous embedding $\mathcal X^p_{\T}\subset\mathcal X^q_{\T}$.
        \item For $\alpha > 1/p$, if $x \in C^{\alpha}([0, 1])$ then $[x]^{(p)}_{\T} \equiv 0$, and in particular $C^{\alpha}([0, 1]) \subset V^p_{\T}$. 
    \end{enumerate}
\end{lemma}

\begin{proof}
    The first result follows immediately from the continuity of $x$:
    \[
        [x]^{(q)}_{\T^n}(t) \le \Big( \max_{j \in I^n} \big|x(t^n_{j + 1} \wedge t) - x(t^n_j \wedge t)\big|\Big)^{q-p} \sum_{j \in I^n} \big \vert x(t^n_{j + 1} \wedge t) - x(t^n_j \wedge t)\big\vert^p \xlongrightarrow{n \to \infty} 0,
    \]
    and the second one uses H\"older inequality:  
    \[
        \big([x]^{(q)}_{\T^n}(1)\big)^{1/q}
        = \Big(\sum_{j\in I_n} \big\vert x(t^n_{j + 1}) - x(t^n_j)\big\vert^q \Big)^{1/q} 
        \le \Big(\sum_{j\in I_n} \big\vert x(t^n_{j + 1}) - x(t^n_j)\big\vert^p \Big)^{1/p}
        = \big([x]^{(p)}_{\T^n}(1)\big)^{1/p}.
    \]
    For the last one, we have for any $t \in [0, 1]$
    \[
        [x]^{(p)}_{\T^n}(t) \le \sum_{j\in I_n} \big|x(t^n_{j+1})-x(t^n_j)\big|^p \le \Vert x \Vert_{C^{\alpha}} \sum_{j\in I_n} (2^{-n\alpha})^p = \Vert x \Vert_{C^{\alpha}} 2^{n(1-\alpha p)} \xrightarrow{n\to\infty} 0.
    \]
\end{proof}

\medskip

\subsection{Faber-Schauder representation}

In this subsection, we briefly review the classical Faber-Schauder system, which was studied in the early 1900s by \citet{faber1910} and later generalized by \citet{schauder1927}.

For the following Haar function defined on $[0, 1]$
\begin{equation*}
    \psi(t) :=
    \begin{cases}
        ~1, &\qquad \text{if } t \in [0, \frac{1}{2}),
        \\
        -1, &\qquad \text{if } t \in [\frac{1}{2}, 1),
        \\
        ~0, &\qquad \text{otherwise},
    \end{cases}
\end{equation*}
we consider the Haar basis for each $m \ge 0$ and $k \in I_m$
\begin{equation}\label{Eq:Haar}
    \psi_{m, k}(t) := 2^{\frac{m}{2}}\psi(2^m t-k), \qquad \forall \, t \in [0, 1].
\end{equation}
Note that $\{ \mathbf{1}_{[0,1)} \} \cup \{\psi_{m, k}\}_{m \ge 0, k \in I_m}$ is an orthonormal basis of $L^2([0, 1])$ with respect to the inner product $\langle f, g \rangle = \int_0^1 f(t)g(t)dt$. The Faber-Schauder functions $\{e_{m, k}\}_{m \ge 0, k \in I_m}$ are defined by integrating the Haar functions $\psi_{m, k}$ for each $m \ge 0$ and $k \in I_m$
\begin{equation}    \label{Eq: e_mk}
    e_{m, k}(t) := \int_0^t \psi_{m, k}(s) \, ds, \qquad t \in [0, 1].
\end{equation}

Since the Faber-Schauder system forms a Schauder basis of $C([0,1])$, every $x \in C([0,1])$ admits a unique Faber-Schauder representation
\begin{equation}    \label{Eq: Schauder representation}
    x(t) = x(0) + \big(x(1)-x(0)\big)t+ \sum_{m=0}^{\infty}\sum_{k \in I_m} \theta^x_{m, k} e_{m, k}(t), \qquad \forall \; t \in [0, 1],
\end{equation}
where the real numbers $\{\theta^x_{m, k}\}_{m \ge 0, k \in I_m}$ are called the \textit{Faber-Schauder coefficients of $x$}. The representation \eqref{Eq: Schauder representation} is an example of wavelet expansion of continuous functions; however, a notable advantage of it is that each coefficient admits a closed-form expression in terms of the function values of $x$ at the dyadic points:
\begin{equation}\label{eq:theta_mk}
    \theta^x_{m, k} = 2^{\frac{m}{2}}\bigg(2 x\Big(\frac{2k+1}{2^{m+1}}\Big)-x\Big(\frac{k}{2^m}\Big) -  x\Big(\frac{k+1}{2^{m}}\Big) \bigg), \qquad m \ge 0, ~ k \in I_m.
\end{equation}

\medskip

\subsection{Faber-Schauder coefficients}

The Faber-Schauder coefficients of continuous functions are closely related to their regularity properties. In this subsection, we collect such preliminary results. All of them are taken from the existing literature and are stated without proof.

In 1960, \citet{Ciesielski:isomorphism} gave a characterization of H\"older continuity in terms of the Faber-Schauder coefficients along the dyadic partition sequence. This result was recently generalized to a wider class of partition sequences \cite[Theorem 3.4]{fake_fBM}.

\begin{lemma} [\citet{Ciesielski:isomorphism}] \label{lem: Ciesielski}
    Suppose that $x \in C([0, 1])$ admits the Faber-Schauder representation \eqref{Eq: Schauder representation}. For any $\alpha \in (0, 1)$, we have $x \in C^{\alpha}([0, 1])$ if and only if 
    \begin{equation}    \label{con: Holder}
    \sup_{m \ge 0, \, k \in I_m}\Big( 2^{m(\alpha-\frac{1}{2})}|\theta^x_{m, k}|\Big) < \infty.
    \end{equation}
\end{lemma}

Next, we recall an equivalent condition on the Faber-Schauder coefficients for a continuous function $x\in C([0,1])$ to belong to the Banach space $\mathcal{X}^p_{\T}$ defined in \eqref{def : x p pi space}. The smallest $p$ such that $x\in \mathcal{X}^p_{\T}$ is called the \textit{variation index} of $x$; it measures the roughness of $x$ in terms of the finiteness of its $p$-th variation \cite{das-kim2024}. The quantity $\xi_m^{(p)}$ defined in \eqref{def : xi general} will play a key role in Section~\ref{subsec: paths with linear p-th variation}.

\begin{lemma} [Dyadic case of Theorem 4.3 of \cite{das-kim2024}]    \label{lem: X^p space}
    For $x \in C([0, 1])$ with the Faber-Schauder representation \eqref{Eq: Schauder representation}, we denote for any $p > 1$
    \begin{equation}    \label{def : xi general}
        \xi^{(p)}_m := 2^{-\frac{mp}{2}} \Big( \sum_{k \in I_m} |\theta^x_{m, k}|^p \Big), \qquad m \ge 0.
    \end{equation}
    Then,
    \begin{equation*}
        x \in \mathcal{X}^p_{\T} \quad \text{if and only if} \quad \limsup_{m \to \infty} \, \xi^{(p)}_m < \infty.
    \end{equation*}
\end{lemma}

\bigskip

\section{Results}   \label{sec: results}

The results of this paper are organized into five closely related topics, each presented in a separate subsection.

\subsection{Paths with linear \texorpdfstring{$p$}{TEXT}-th variation}  \label{subsec: paths with linear p-th variation}

We first construct continuous functions with linear dyadic $p$-th variation for any $p>1$ by imposing a structural condition on their Faber-Schauder coefficients. The idea is that, if the coefficients have the same magnitude within each dyadic level, then the contribution of that level to the discrete $p$-th variation is distributed uniformly across dyadic subintervals. This leads to the following condition.

\begin{assumption}[Uniform magnitude condition] \label{assum: uniform mag}
    We say that a real sequence $(\theta_{m,k})_{m\ge0,\,k\in I_m}$ satisfies the \emph{uniform magnitude condition}, if there exist a sequence $(c_m)_{m\ge0}$ of nonnegative numbers and a sign array $(\sigma_{m,k})_{m\ge0,\,k\in I_m}\subseteq\{\pm1\}$ such that
    \begin{equation}\label{eq:TL-signed-coeff}
        \theta_{m,k}=c_m \sigma_{m,k} \quad \text{holds for every } k \in I_m, ~ m \ge 0.
    \end{equation}
\end{assumption}
\noindent Given a coefficient array $(\theta_{m,k})_{m\ge0,\,k\in I_m}$, we define for $p>1$ as in the notation \eqref{def : xi general}
\[
    \xi_m^{(p)} := 2^{-mp/2}\sum_{k\in I_m} |\theta_{m,k}|^p, \qquad m\ge0,
\]
and consider the Faber-Schauder series
\begin{equation}    \label{eq: FS repre}
    x(t)= \sum_{m=0}^{\infty} \sum_{k\in I_m} \theta_{m,k} e_{m,k}(t), \qquad t \in[0,1],
\end{equation}
whenever the series converges. For simplicity, compared to the Faber-Schauder representation \eqref{Eq: Schauder representation}, we restrict to the case $x(0)=x(1)=0$; the general case follows by adding an affine function, which does not affect the $p$-th variation. Under Assumption~\ref{assum: uniform mag}, convergence of $\xi_m^{(p)}$ turns out to be sufficient for $x$ to have linear $p$-th variation. Its proof is lengthy, so it is deferred to Section~\ref{subsec: proof main}.

\begin{theorem} \label{thm: xi implies p-th variation}
    Fix $p>1$. Suppose that a sequence $(\theta_{m,k})_{m\ge0,\,k\in I_m}$ satisfying the uniform magnitude condition of Assumption~\ref{assum: uniform mag} is given. If the sequence $\xi^{(p)}_n$ defined via \eqref{def : xi general} converges as $n \to \infty$, then $x$ of \eqref{eq: FS repre} is a $\frac 1p$-H\"older continuous function with linear $p$-th variation along $\T$. More precisely, there exists a positive constant $C_p$, which depends only on $p$, satisfying
    \begin{equation}    \label{eq: xi implies p-th variation}
        \lim_{n\to\infty} [x]^{(p)}_{\T^n}(t) = C_p t \lim_{n\to\infty} \xi^{(p)}_n, \qquad \forall \; t \in [0, 1].
    \end{equation}
\end{theorem}

Theorem~\ref{thm: xi implies p-th variation} provides a concrete coefficient-level criterion for producing reference paths with linear dyadic $p$-th variation. Other approaches to construct paths with linear $p$-th variation have been studied in \cite{Misura2019, schied2020fractal}. These paths will play a central role in the multiplicative construction developed in the next subsection.

\medskip

\subsection{Constructing paths with prescribed 
\texorpdfstring{$p$}{TEXT}-th variation} \label{subsec: translation}

This subsection develops a multiplicative construction of continuous paths with prescribed, possibly nonlinear $p$-th variation, starting from the linear $p$-th variation paths built in the previous subsection. The basic observation is that multiplying a path in $V^p_{\T}$ by a path with vanishing $p$-th variation preserves membership in $V^p_{\T}$ and transforms the $p$-th variation in an explicit way.

\begin{proposition}   \label{prop: translation}
    For any $p>1$ suppose that $x\in V^p_{\T}$ and $g\in V^p_{\T}$ with $[g]^{(p)}_{\T} \equiv 0$ are given. Then, the continuous function $y := g x \in C([0, 1])$ also belongs to $V^p_{\T}$, and admits finite dyadic $p$-th variation
    \begin{equation}    \label{eq: p-th variation of y}
        [y]^{(p)}_{\T}(t) = \int_0^t |g(u)|^p \, d[x]^{(p)}_{\T}(u), \qquad \forall \, t \in [0, 1].
    \end{equation}
\end{proposition}

The proof of Proposition~\ref{prop: translation} is given in Section~\ref{subsec: proof translation}. Proposition~\ref{prop: translation} provides the basic mechanism for constructing paths with $p$-th variation: once a reference path with known $p$-th variation is fixed, one only needs to choose a suitable multiplier. The following theorem uses this mechanism to construct continuous paths with a given (nonlinear) $p$-th variation. The role of the hypothesis on $(h')^{1/p}$ in the theorem will become clear from the proof.

\begin{theorem} \label{thm: recipe}
    Fix any $p>1$. Suppose that $h\in C^1([0,1])$ is non-decreasing with $h(0)=0$, and that $\big[(h')^{1/p}\big]^{(p)}_{\T}\equiv 0$. Then, there exists a function $y\in V^p_{\T}$ such that
    \begin{equation}    \label{eq: y p-th variation}
        [y]^{(p)}_{\T}(t)=h(t),\qquad \forall\,t\in[0,1].
    \end{equation}
    Moreover, if in addition $(h')^{1/p}\in C^{\alpha}([0,1])$ for some $\alpha\leq \frac1p$, then $y$ is also in $C^{\alpha}([0,1])\cap V^p_{\T}$.
\end{theorem}

\begin{proof}
    Let $(\theta_{m,k})$ satisfy Assumption~\ref{assum: uniform mag} with $c_m=2^{m(\frac12-\frac1p)}$, and let $x$ be the corresponding path in \eqref{eq: FS repre}. Then $\xi_m^{(p)}\equiv 1$, $x$ is $\frac1p$-H\"older continuous by Lemma~\ref{lem: Ciesielski}, and Theorem~\ref{thm: xi implies p-th variation} yields $[x]^{(p)}_{\T}(t)=C_p t$ for all $t\in[0,1]$.

    Set
    \begin{equation}    \label{def: g function}
        g(t):=\left(\frac{h'(t)}{C_p}\right)^{1/p},\qquad t\in[0,1].
    \end{equation}
    By assumption, $[g]^{(p)}_{\T}\equiv 0$ (thus $g\in V^p_{\T}$). Hence, applying Proposition~\ref{prop: translation} to $y(t):=g(t)x(t)$ gives
    \[
        [y]^{(p)}_{\T}(t)
        =\int_0^t |g(u)|^p\,d[x]^{(p)}_{\T}(u)
        =C_p\int_0^t |g(u)|^p\,du
        =\int_0^t h'(u)\,du
        =h(t),\qquad \forall \, t\in[0,1].
    \]
    If in addition $(h')^{1/p}\in C^{\alpha}([0,1])$ for some $\alpha\leq\frac1p$, then $g\in C^{\alpha}([0,1])$. Since $x \in C^{\frac1p}([0,1])$, their product $y$ belongs to $C^{\min\{\alpha,\frac1p\}}([0,1]) = C^\alpha([0,1])$.
\end{proof}

\begin{remark}
    In \cite[Example 1]{fake_fBM}, one can construct paths with H\"older exponent strictly smaller than $\frac13$ with vanishing quadratic variation. Choosing this function as $g$ in \eqref{def: g function} with $p=2$, Theorem~\ref{thm: recipe} yields a path $y\in C^{1/3}([0,1])\cap V^2_{\T}$ whose quadratic variation is nontrivial and is given by \eqref{eq: y p-th variation}. This path $y$ is especially interesting as one cannot apply the typical rough It\^o formula \cite{FrizHairer} for paths with H\"older exponent strictly smaller than $\frac13$, whereas F\"ollmer's pathwise It\^o theory remains applicable.
\end{remark}

The next remark and example provide a few sufficient conditions and examples for $h$ to satisfy the hypotheses of Theorem~\ref{thm: recipe}.

\begin{remark}     \label{rem:recipe-sufficient}
    Let $p>1$ and $h\in C^1([0,1])$ be non-decreasing with $h(0)=0$. In view of Lemma~\ref{lem:inclusions}, the hypothesis $\big[(h')^{1/p}\big]^{(p)}_{\T}\equiv 0$ in Theorem~\ref{thm: recipe} holds, for instance, under any of the following conditions:
    \begin{enumerate}[(i)]
        \item $(h')^{1/p} \in V^q_{\T}$ for some $q < p$, or a stronger sufficient condition is that $(h')^{1/p}$ has bounded variation on $[0,1]$;
        \item $(h')^{1/p}\in C^{\alpha}([0,1])$ for some $\alpha>1/p$.
    \end{enumerate}
    Moreover, the additional assumption $(h')^{1/p}\in C^{\alpha}([0,1])$ for some $\alpha\leq \frac1p$ in Theorem~\ref{thm: recipe} holds, for instance, if $h'$ is Lipschitz on $[0,1]$.
\end{remark}

\begin{example}[Concrete choices of $h$]    \label{ex:recipe-concrete}
    Fix $p>1$. Each of the following functions $h$ satisfies the hypotheses of Theorem~\ref{thm: recipe}, and also the additional H\"older condition $(h')^{1/p}\in C^{\alpha}([0,1])$ for some $\alpha \le \frac1p$; indeed, in all cases, $h\in C^2([0,1])$, hence $h'$ is Lipschitz on $[0,1]$, which implies $(h')^{1/p}\in C^{\alpha}([0,1])$.
    \begin{enumerate}[(1)]
        \item Polynomials with nonnegative coefficients: for $M \in \mathbb{N}$, $h(t)=\sum_{k=1}^M a_k t^k$ with $a_k\ge 0$.
        \item Exponential growth: for $a>0$, $h(t)=e^{at}-1$.
        \item Logarithmic growth: for $b>0$, $h(t)=\log(1+bt)$.
        \item Rational function: $h(t)=\frac{t}{1+t}$.
        \item Arctangent function: for $b>0$, $h(t)=\arctan(bt)$.
    \end{enumerate}
\end{example}

Theorem~\ref{thm: recipe} shows that for any $p>1$ the space $V^p_{\T}\cap C^{\frac1p}([0,1])$ is nonempty, and its proof provides an explicit construction. The following remark discusses the spaces $V^p_{\T}\cap C^{\alpha}([0,1])$ depending on the value of $\alpha p$.

\begin{remark}  \label{rem: other spaces}
    The remaining cases are as follows.
    \begin{enumerate}
        \item [(i)] \textbf{Case $\alpha p< 1$}: Let $y\in V^p_{\T}\cap C^{\frac1p}([0,1])$ be the path constructed in Theorem~\ref{thm: recipe}. Since $C^{\beta}([0,1])\subset C^{\alpha}([0,1])$ for $0<\alpha\le \beta\le 1$, we have $y\in V^p_{\T}\cap C^{\alpha}([0,1])$ for every $\alpha\le \frac1p$. In particular, $V^p_{\T}\cap C^{\alpha}([0,1])$ is nonempty for $\alpha p < 1$ and contains paths with nontrivial $p$-th variation as in \eqref{eq: y p-th variation}.

        \item [(ii)] \textbf{Case $\alpha p > 1$}: By Lemma~\ref{lem:inclusions} (iii), if $x\in C^{\alpha}([0,1])$ with $\alpha>\frac1p$, then $[x]^{(p)}_{\T}\equiv 0$. Consequently, $V^p_{\T}\cap C^{\alpha}([0,1])$ is nonempty for $\alpha p>1$ (it contains every path in $C^\alpha([0, 1])$), but every element has vanishing $p$-th variation along $\T$.
    \end{enumerate}
\end{remark}

We conclude this subsection with a brief note that, for $p=2$, \citet{Misura2019} obtained related deterministic constructions of continuous functions with prescribed pathwise quadratic variation. Their approach combines dyadic/Faber-Schauder descriptions of quadratic variation with, in the local case, a more specialized construction based on pathwise It\^o differential equations of Doss-Sussman type. By contrast, our approach here is more direct and is tailored to the dyadic $p$-th variation setting for general $p>1$: starting from a reference path with linear dyadic $p$-th variation, we multiply by functions of vanishing $p$-th variation to obtain explicit formulas for the resulting prescribed variation. This same mechanism will also underlie the density and transported Banach space results developed in the subsequent subsections.

\medskip

\subsection{Paths with prescribed \texorpdfstring{$p$}{TEXT}-th variation are dense in \texorpdfstring{$C([0, 1])$}{TEXT}}  \label{subsec: dense}

Inspired by Section~\ref{subsec: translation}, we now show that any continuous path in $C([0,1])$ can be approximated arbitrarily close in uniform norm by a path with prescribed $p$-th variation along the dyadic partition sequence. Hence, for any $p>1$, every nonempty class of continuous paths with a fixed prescribed $p$-th variation is dense in $C([0,1])$.

Let $\tau : [0,1] \rightarrow [0,\infty)$ be any continuous, non-decreasing function with $\tau(0) = 0$. Given such $\tau$, we consider the subset of $V^p_\T$ that has $p$-th variation equal to $\tau$:
\begin{equation}    \label{def: tau p-th variation}
    V^{p,\tau}_{\T} := \big\{x \in V^p_\T : [x]^{(p)}_\T = \tau \big\} \subset  V^p_\T.
\end{equation}
For several classes of such functions $\tau$, nonemptiness of $V^{p,\tau}_{\T}\cap C^\alpha([0,1])$ was established in Section~\ref{subsec: translation}; see Theorem~\ref{thm: recipe}, Remark~\ref{rem:recipe-sufficient}, and Remark~\ref{rem: other spaces}. The next theorem shows that every such nonempty class is in fact dense in $C([0,1])$. Its proof, given in Section~\ref{subsec: proof density}, combines the construction from Section~\ref{subsec: translation} with polynomial approximation, implemented via Bernstein polynomials.

\begin{theorem}  \label{thm:densityoffnsinVp}
    Fix $p>1$, $\alpha\in(0,1)$, and a continuous, non-decreasing function $\tau$ defined on $[0, 1]$ with $\tau(0)=0$ such that $V^{p,\tau}_{\T}\cap C^\alpha([0,1])\neq\emptyset$. Then $V^{p,\tau}_{\T}\cap C^\alpha([0,1])$ is dense in $C([0,1])$.
\end{theorem}

\medskip

\subsection{Banach subspaces of \texorpdfstring{$V^p_{\T}$}{TEXT} via multiplicative transports}       \label{subsec: Banach subspaces}

We now use the multiplicative construction from Section~\ref{subsec: translation} to embed Banach spaces of vanishing $p$-th variation paths into the nonlinear space $V^p_{\T}$. To make the transport map invertible, we shall work with a strictly positive reference path obtained by shifting a path with linear dyadic $p$-th variation. Fix $p>1$, and let $x\in V^p_{\T}$ be a path with linear dyadic $p$-th variation such that, for some constant $C_p>0$
\begin{equation*}
    [x]^{(p)}_{\T}(t)=C_p t,\qquad \forall \, t\in[0,1].
\end{equation*}
For example, the function $x$ constructed in the proof of Theorem~\ref{thm: recipe} satisfies this condition. Choose a constant $M>\|x\|_\infty$ and set
\begin{equation}    \label{def: shifted x}
    \bar x:=x+M,
\end{equation}
so that $\bar x\in C([0,1])$ satisfies $\inf_{t\in[0,1]}\bar x(t)\ge M-\|x\|_\infty>0$, and
\begin{equation}    \label{eq: linear variation x bar}
    [\bar x]^{(p)}_{\T}(t)=[x]^{(p)}_{\T}(t)=C_p t, \qquad \forall \, t\in[0,1].
\end{equation}

Building on the multiplicative construction from Section~\ref{subsec: translation} and applying it to the shifted path $\bar x$, the next theorem produces Banach subspaces of $V^p_{\T}$ consisting of paths with explicitly controlled $p$-th variation.

\begin{theorem}[Transported Banach subspaces]   \label{thm: transported Banach}
    Fix $p>1$. For any path $x \in V^p_{\T}$ with linear dyadic $p$-th variation, consider its positive shift $\bar x$, given by \eqref{def: shifted x}, satisfying \eqref{eq: linear variation x bar}. Let $(B,\|\cdot\|_{B})$ be a Banach space such that $B \subset V^{p, 0}_{\T}$ in the notation of \eqref{def: tau p-th variation}, i.e.,
    \begin{equation}    \label{con: Banach B}
        B \subset V^p_{\T} \qquad \text{and} \qquad [g]^{(p)}_{\T}\equiv 0 ~ \text{ for every } g\in B.
    \end{equation}
    Define $\mathcal L_{\bar x}(B) := \{ g \bar{x}: g\in B \} \subset C([0,1])$ and equip $\mathcal L_{\bar x}(B)$ with the norm
    \begin{equation}    \label{eq: Lx norm}
        \|y\|_{\mathcal L} := \|g\|_B = \Big\|\frac{y}{\bar x}\Big\|_{B}, \qquad y\in\mathcal L_{\bar x}(B).
    \end{equation}
    Then, the map $T : B \to \mathcal L_{\bar x}(B)$ defined by $T(g)=g\bar x$ is a linear isometric isomorphism. In particular, $\big(\mathcal L_{\bar x}(B),\|\cdot\|_{\mathcal L}\big)$ is a Banach space. Moreover, $\mathcal{L}_{\bar x}(B) \subset V^p_{\T}$ and for every $y=g\bar x\in\mathcal L_{\bar x}(B)$
    \begin{equation}    \label{eq: pvar on Lx}
        [y]^{(p)}_{\T}(t)=\int_0^t |g(u)|^p\,d[\bar x]^{(p)}_{\T}(u) = C_p\int_0^t |g(u)|^p\,du,\qquad \forall \, t\in[0,1].
    \end{equation}
\end{theorem}

\begin{proof}
    Every $y\in\mathcal L_{\bar x}(B)$ can be written as $y=g\bar x$ for some $g\in B$ by definition. Since $\inf\bar x>0$, this representation is unique and $g=y/\bar x$. Thus, the map $T(g)=g\bar x$ is bijective and linear.
    
    For $g\in B$, we have by definition \eqref{eq: Lx norm} that
    \[
        \|T(g)\|_{\mathcal L}=\Big\|\frac{g\bar x}{\bar x}\Big\|_{B} = \|g\|_{B},
    \]
    so $T$ is an isometry. Since $B$ is complete, $\mathcal L_{\bar x}(B)$ is complete as well.

    Finally, fix $g\in B \subset V^p_{\T}$. Since we have $[g]^{(p)}_{\T}\equiv 0$ and $\bar x\in V^p_{\T}$ from \eqref{eq: linear variation x bar}, applying Proposition~\ref{prop: translation} to the product $y=g\bar x$ yields \eqref{eq: pvar on Lx}.
\end{proof}

For $p>1$, it is easy to check that $V^{p,0}_\T$ is a vector space. In these terms, Theorem~\ref{thm: transported Banach} shows that every Banach space $B\subset V^{p,0}_\T$ can be transported isometrically into $V^p_{\T}$ by multiplication with the fixed positive reference path $\bar x$, and that the $p$-th variation on the transported space $\mathcal L_{\bar x}(B)$ is given explicitly by \eqref{eq: pvar on Lx}.

Since the map $T$ in Theorem~\ref{thm: transported Banach} is a linear isometric isomorphism, basic Banach space properties of $B$ are preserved by the transport, as in the following corollary.
\begin{corollary}   \label{cor: inherited properties}
    In the setting of Theorem~\ref{thm: transported Banach}, $\mathcal L_{\bar x}(B)$ is separable (resp.\ reflexive) if and only if $B$ is separable (resp.\ reflexive).
\end{corollary}

We now give some concrete and familiar choices of the source Banach space $(B,\|\cdot\|_B)$ satisfying \eqref{con: Banach B}, thereby producing explicit Banach subspaces of $V^p_{\T}$ through Theorem~\ref{thm: transported Banach}.

\begin{example} \label{ex: Banach examples}
    For any $p>1$, each of the following Banach spaces $(B,\|\cdot\|_B)$ satisfies \eqref{con: Banach B}; hence Theorem~\ref{thm: transported Banach} applies.
    \begin{enumerate}[(i)]
        \item \textbf{H\"older spaces.}
        For any $\alpha> \frac 1p$, let $B = C^\alpha([0,1])$ equipped with the H\"older norm $\|\cdot\|_{C^\alpha}$ (see Lemma~\ref{lem:inclusions} (iii)).
        \item \textbf{One-dimensional Sobolev spaces.}
        For any $r > p' := \frac{p}{p-1}$, let $B = W^{1,r}([0,1])$ equipped with the standard Sobolev norm, identifying each Sobolev class with its continuous representative. In fact, this space is a special case of (i), as $W^{1, r}([0, 1]) \hookrightarrow C^{1-1/r}([0, 1])$.
        \item \textbf{Continuous, bounded variation functions.}
        Let $B = BV([0,1]) \cap C([0,1])$ equipped with the norm $\|g\|_{\infty} + \|g\|_{TV}$, where 
        \[
            \|g\|_{TV} := \sup\Big\{\sum_{i=0}^{m-1}\big|g(t_{i+1})-g(t_i)\big| \,:\, 0=t_0<t_1<\cdots<t_m=1,\ m\in\mathbb{N}\Big\}\in[0,\infty].
        \]
        is the total variation of $g$ on $[0, 1]$.
    \end{enumerate}
    In each case, $\mathcal L_{\bar x}(B)$ is a Banach subspace of $V^p_{\T}$ under the transported norm $\|\cdot\|_{\mathcal L}$ in \eqref{eq: Lx norm}, and we have the explicit representation \eqref{eq: pvar on Lx} of the $p$-th variation for every $y\in\mathcal L_{\bar x}(B)$.
\end{example}

\begin{remark} [An equivalent norm incorporating the $p$-th variation]   \label{rem: equiv norm}
    Under the assumptions of Theorem~\ref{thm: transported Banach}, the identity \eqref{eq: pvar on Lx} gives
    \[
        \big([y]^{(p)}_{\T}(1)\big)^{\frac 1p}=(C_p)^{\frac 1p}\|g\|_{L^p([0,1])} \qquad \text{for every } y = g \bar{x} \in \mathcal L_{\bar x}(B).
    \]
    Hence, if the embedding $B \hookrightarrow L^p([0,1])$ is continuous (as in the three examples of Example~\ref{ex: Banach examples}), then the norm
    \[
        \|y\|_{\mathcal L, p} := \|g\|_{B} + (C_p)^{\frac 1p}\|g\|_{L^p([0,1])}, \qquad \text{where } y=g\bar x\in\mathcal L_{\bar x}(B).
    \]
    is equivalent to $\|\cdot\|_{\mathcal L}$ on $\mathcal L_{\bar x}(B)$.
\end{remark}

\medskip

\subsection{Stability of pathwise It\^o maps on transported subspaces}\label{subsec: stability Ito}

In this subsection, we study stability properties of transported Banach subspaces $\mathcal L_{\bar x}(B)$. Throughout, we work under the setting of Theorem~\ref{thm: transported Banach} and assume in addition that the embedding $B\hookrightarrow L^p([0,1])$ is continuous, as in Example~\ref{ex: Banach examples} and Remark~\ref{rem: equiv norm}. This extra assumption allows us to quantify how the prescribed dyadic $p$-th variation changes under perturbations in the transported norm.

For every $y=g\bar x\in\mathcal L_{\bar x}(B)$, the $p$-th variation measure is absolutely continuous:
\[
    d[y]^{(p)}_{\T}(u)=C_p|g(u)|^p\,du.
\]
Consequently, in any F\"ollmer-type pathwise It\^o formula whose correction term involves integration against $d[y]^{(p)}_{\T}$, that term reduces to a Lebesgue integral with the explicit density $C_p|g|^p$.
For instance, in F\"ollmer's classical quadratic-variation case \cite{follmer1981}, for $f\in C^2(\R)$ one obtains
\begin{align}
    f(y_t) &= f(y_0) + \int_0^t f'(y_u)\,dy_u + \frac12\int_0^t f''(y_u)\,d[y]^{(2)}_{\T}(u) \label{eq: Follmer classical formula}
    \\
    &= f(y_0) + \int_0^t f'(y_u)\,dy_u + \frac{C_2}{2}\int_0^t f''(y_u)|g(u)|^2\,du.    \nonumber
\end{align}
Here, we denote $C^k(\mathbb{R})$ for $k \in \mathbb{N}$ the space of functions that are $k$ times differentiable with continuous $k$-th order derivative.

We first describe stability of the prescribed $p$-th variation mapping on $\mathcal L_{\bar x}(B)$.

\begin{proposition}[Stability of the prescribed $p$-th variation]   \label{prop: Gamma-stability}
    In the setting of Theorem~\ref{thm: transported Banach}, assume that the embedding
    $B\hookrightarrow L^p([0,1])$ is continuous. Consider the mapping
    \begin{equation}    \label{eq: mappping}
        y \mapsto [y]^{(p)}_{\T}
    \end{equation}
    from $(\mathcal L_{\bar x}(B), \|\cdot\|_{\mathcal L})$ to $(C([0,1]), \|\cdot\|_{\infty})$. Then the following statements hold.
    \begin{enumerate}[(i)]
        \item For $y_i=g_i\bar x\in\mathcal L_{\bar x}(B)$, $i=1,2$,
        \begin{equation}\label{eq:Gamma-L1-bound}
            \big\| [y_1]^{(p)}_{\T} - [y_2]^{(p)}_{\T} \big\|_{\infty}
            \le C_p\,\big\||g_1|^p-|g_2|^p\big\|_{L^1}.
        \end{equation}
        \item If $y_n\to y$ in $(\mathcal L_{\bar x}(B),\|\cdot\|_{\mathcal L})$, then
        $[y_n]^{(p)}_{\T} \to [y]^{(p)}_{\T}$ uniformly on $[0,1]$.
        \item The mapping \eqref{eq: mappping} is locally Lipschitz. More precisely, if $K_{B,p}>0$ is such that $\|g\|_{L^p}\le K_{B,p}\|g\|_B$ for $g\in B$, then for $y_i=g_i\bar x$
        \begin{equation}\label{eq:Gamma-local-Lip}
            \big\| [y_1]^{(p)}_{\T} - [y_2]^{(p)}_{\T} \big\|_{\infty}
            \le p C_p K_{B,p}^{p} \Big(\|y_1\|_{\mathcal L}^{p-1}+\|y_2\|_{\mathcal L}^{p-1}\Big)\,\|y_1-y_2\|_{\mathcal L}.
        \end{equation}
        In particular, on each ball $\{y:\ \|y\|_{\mathcal L}\le R\}$, \eqref{eq: mappping} is Lipschitz with constant
        $2pC_p K_{B,p}^p R^{p-1}$.
    \end{enumerate}
\end{proposition}

\begin{proof}
    Let $y_i=g_i\bar x$, $i=1,2$. Using the identity \eqref{eq: pvar on Lx}, we have for all $t\in[0,1]$
    \[
        \big|[y_1]^{(p)}_{\T}(t)-[y_2]^{(p)}_{\T}(t)\big|
        \le C_p\int_0^1 \big||g_1(u)|^p-|g_2(u)|^p\big|\,du,
    \]
    which yields \eqref{eq:Gamma-L1-bound}.

    For (ii), if $y_n=g_n\bar x\to y=g\bar x$ in $(\mathcal L_{\bar x}(B), \|\cdot\|_{\mathcal L})$, then $g_n\to g$ in $(B, \|\cdot\|_{B})$, thanks to the isometric isomorphism $T$ of Theorem~\ref{thm: transported Banach}. By the assumed continuous embedding $B\hookrightarrow L^p([0,1])$, we have $g_n\to g$ in $L^p([0,1])$, hence $\||g_n|^p-|g|^p\|_{L^1}\to 0$. Therefore, \eqref{eq:Gamma-L1-bound} implies
    $\|[y_n]^{(p)}_{\T}-[y]^{(p)}_{\T}\|_\infty\to 0$.

    For (iii), using the inequality
    \[
        \big||a|^p-|b|^p\big|\le p\big(|a|^{p-1}+|b|^{p-1}\big)|a-b|,
        \qquad a,b\in\R,
    \]
    with H\"older inequality, we obtain
    \[
        \big\||g_1|^p-|g_2|^p\big\|_{L^1}
        \le p\Big(\|g_1\|_{L^p}^{p-1}+\|g_2\|_{L^p}^{p-1}\Big)\|g_1-g_2\|_{L^p}.
    \]
    Finally, by the continuous embedding $B\hookrightarrow L^p([0, 1])$ and $\|y\|_{\mathcal L}=\|g\|_B$ for $y=g\bar x$,
    \[
        \|g_i\|_{L^p}\le K_{B,p}\|g_i\|_B = K_{B,p}\|y_i\|_{\mathcal L},
        \qquad
        \|g_1-g_2\|_{L^p}\le K_{B,p}\|g_1-g_2\|_B = K_{B,p}\|y_1-y_2\|_{\mathcal L}.
    \]
    Combining the estimates yields \eqref{eq:Gamma-local-Lip}.
\end{proof}

When $p$ is an even integer, we have the direct generalization of \eqref{eq: Follmer classical formula} from \cite{perkowski2019}; for $y\in V^p_{\T}$ and $f\in C^p(\mathbb R)$, the pathwise F\"ollmer-It\^o integral is defined as
\begin{equation}    \label{def: Rama Nicolas}
    \int_0^t f'(y_u)dy_u := \lim_{n\to\infty}\sum_{t_i^n<t}\sum_{k=1}^{p-1} \frac{f^{(k)}\big(y(t_i^n)\big)}{k!}\Big(y(t_{i+1}^n\wedge t)-y(t_i^n\wedge t)\Big)^k, \qquad t \in [0, 1],
\end{equation}
whenever the limit exists. Using the pathwise change-of-variable formula of \cite{perkowski2019}, we prove a continuity statement for the pathwise F\"ollmer-It\^o map for even integers $p$ on transported subspaces.

\begin{theorem}[Continuity of the F\"ollmer-It\^o map on $\mathcal L_{\bar x}(B)$ for even $p$] \label{thm: Ito map continuity even p}
    Assume the setting of Theorem~\ref{thm: transported Banach}, where $p\in 2\mathbb N$.
    Suppose in addition that the embeddings $B\hookrightarrow C([0,1])$ and $B\hookrightarrow L^p([0,1])$ are continuous. For $f\in C^p(\mathbb R)$, define the F\"ollmer-It\^o map via \eqref{def: Rama Nicolas}
    \[
        \mathcal I_f:\mathcal L_{\bar x}(B)\to C([0,1]),\qquad
        \mathcal I_f(y)(t):=\int_0^t f'(y_u)\,dy_u.
    \]
    Then $\mathcal I_f$ is continuous with respect to the norm $\|\cdot\|_{\mathcal L}$ on $\mathcal L_{\bar x}(B)$
    and the uniform norm on $C([0,1])$.
\end{theorem}

\begin{proof}
    Let $y_n=g_n\bar x$ and $y=g\bar x$ with $y_n\to y$ in $(\mathcal L_{\bar x}(B),\|\cdot\|_{\mathcal L})$, i.e., $g_n\to g$ in $(B,\|\cdot\|_B)$. By the assumed embedding $B\hookrightarrow C([0,1])$, we have $g_n\to g$ uniformly on $[0,1]$,
    hence also
    \[
        y_n\to y \qquad \text{uniformly on } [0,1],
    \]
    since $\bar x$ is bounded.

    By the pathwise change-of-variable formula \cite{perkowski2019} for even $p$,
    \begin{align*}
        \mathcal I_f(y_n)(t)
        &= f\big(y_n(t)\big)-f\big(y_n(0)\big)
        -\frac{1}{p!}\int_0^t f^{(p)}\big(y_n(u)\big)\,d[y_n]^{(p)}_{\T}(u)
        \\
        &= f\big(y_n(t)\big)-f\big(y_n(0)\big)
        -\frac{C_p}{p!}\int_0^t f^{(p)}\big(y_n(u)\big)\,|g_n(u)|^p\,du, \qquad t \in [0, 1],
    \end{align*}
    and similarly for $\mathcal I_f(y)(t)$. Since $y_n\to y$ uniformly and $f$ is continuous, the term $f(y_n(\cdot))-f(y_n(0))$ converges uniformly to $f(y(\cdot))-f(y(0))$. For the correction term, write for $t\in[0,1]$,
    \begin{align*}
        &\Big|\int_0^t f^{(p)}\big(y_n(u)\big)|g_n(u)|^p\,du
        -\int_0^t f^{(p)}\big(y(u)\big)|g(u)|^p\,du\Big|
        \\
        &\qquad\le
        \int_0^1 \Big|f^{(p)}\big(y_n(u)\big)-f^{(p)}\big(y(u)\big)\Big|\,|g_n(u)|^p\,du
        +\int_0^1 \Big|f^{(p)}\big(y(u)\big)\Big|\,
        \big||g_n(u)|^p-|g(u)|^p\big|\,du.
    \end{align*}
    Since $f^{(p)}$ is bounded and $g_n\to g$ in $B\hookrightarrow L^p([0,1])$, the sequence $\{g_n\}$ is bounded in $L^p([0,1])$. Moreover, $y_n\to y$ uniformly implies $f^{(p)}(y_n)\to f^{(p)}(y)$ uniformly, so the first integral tends to $0$. The second integral also tends to $0$ because $f^{(p)}$ is bounded and $\||g_n|^p-|g|^p\|_{L^1([0,1])}\to 0$ by the $L^p$ convergence of $g_n\to g$. Therefore, the correction terms converge uniformly in $t\in[0,1]$.

    Combining the uniform convergence of both terms yields
    \[
        \|\mathcal I_f(y_n)-\mathcal I_f(y)\|_\infty\to 0.
    \]
\end{proof}

\begin{remark}[The noninteger case] \label{rem: noninteger Ito}
    When $p>1$ is not an integer, a similar continuity result also holds for the fractional pathwise integral of \cite{ruhong2022}, provided one works in the \emph{no-remainder regime}. More precisely, if $p=m+\alpha$ with $m=\lfloor p\rfloor$ and the functions $f$ and $y \in \mathcal L_{\bar x}(B)$ satisfy the assumptions considered in \cite[Theorem~2.6 and Theorem~2.12]{ruhong2022}, then the compensated Riemann sums
    \[
        \sum_{t_i^n<t}\sum_{k=1}^{m} \frac{f^{(k)}\big(y(t_i^n)\big)}{k!}\Big(y(t_{i+1}^n\wedge t)-y(t_i^n\wedge t)\Big)^k
    \]
    converge to a pathwise integral $\int_0^t f'(y_u)\,dy_u$ and satisfy a change-of-variable formula without an additional $p$-th variation remainder term. In that case, continuity of the associated pathwise It\^o map on $\mathcal L_{\bar x}(B)$ follows by an even simpler argument based only on the uniform convergence $y_n\to y$: for each $t \in[0, 1]$,
    \[
        \int_0^t f'\big(y_n(u)\big)dy_n(u) = f\big(y_n(t)\big)-f\big(y_n(0)\big) \xlongrightarrow{n \to \infty} f\big(y(t)\big)-f\big(y(0)\big) = \int_0^t f'\big(y(u)\big)dy(u).
    \]
\end{remark}

\bigskip

\section{Extensions to more general refining partition sequences}   \label{sec: extention to general partitions}

For simplicity of exposition, we have worked throughout with the dyadic partition sequence until now, but all the results of the previous sections can be extended to a more general class of partition sequences, namely $q$-refining partition sequences (see Definition~\ref{Def: q-adically refining}).

In Section~\ref{subsec: q-adic-linear-pvar}, we first explore the $q$-adic partition sequence, when a parent interval at level $n$ is equally divided into $q$ subintervals of length $1/q^{n+1}$ at level $n+1$, for any fixed integer $q\ge2$.

We next apply the time-change technique to those $q$-adic partition sequences to obtain $q$-refining partition sequences in Section~\ref{subsec:time-change}.

\medskip

\subsection{The \texorpdfstring{$q$}{TEXT}-adic partition sequences}    \label{subsec: q-adic-linear-pvar}

We first introduce notation for $q$-adic partition sequences and the corresponding Haar/Schauder functions. Fix an integer $q\ge2$, and consider the $q$-adic partition sequence
\begin{equation}    \label{def: q-adic partition}
    \T_q=(\T_q^n)_{n\ge0},
    \qquad
    \T_q^n:=\Big\{0,\frac1{q^n},\frac2{q^n},\dots,1\Big\}.
\end{equation}
For $m\ge0$ and $k=0,\dots,q^m-1$, write
\[
    I_{m,k}:=\Big[\frac{k}{q^m},\frac{k+1}{q^m}\Big),
\]
and denote by
\[
    I_{m+1,qk+d}, \qquad d=0,\dots,q-1,
\]
the $q$ children of $I_{m,k}$.

For $\ell=1,\dots,q-1$ and $d=0,\dots,q-1$, define
\[
    \gamma_{\ell,d}:=
    \begin{cases}
        \displaystyle \sqrt{\frac{q}{\ell(\ell+1)}}, & 0\le d\le \ell-1,
        \\[1.2ex]
        \displaystyle -\sqrt{\frac{q\ell}{\ell+1}}, & d=\ell,
        \\[1.2ex]
        ~~~~~0, & d\ge \ell+1.
    \end{cases}
\]
From the definitions of generalized Haar/Schauder functions for a general class of partition sequences \cite[Definitions~3.6, 3.7]{das-kim2024}, the associated $q$-adic Haar functions are given by
\begin{equation}    \label{def: q-adic Haar}
    \psi_{m,k,\ell}(t) := q^{m/2}\sum_{d=0}^{q-1}\gamma_{\ell,d}\,\mathbbm{1}_{I_{m+1,qk+d}}(t), \qquad t \in [0,1],
\end{equation}
and the corresponding generalized Schauder functions are
\[
    e_{m,k,\ell}(t):=\int_0^t \psi_{m,k,\ell}(s)\,ds,
    \qquad t\in[0,1].
\]

Let $a=(a_1,\dots,a_{q-1})\in\mathbb R^{q-1}\setminus\{0\}$ be fixed, and define
\begin{equation}    \label{def: eta_d}
    \eta_d(a):=\sum_{\ell=1}^{q-1} a_\ell \gamma_{\ell,d}, \qquad d=0,\dots,q-1.
\end{equation}

The key result in this subsection is to generalize Theorem~\ref{thm: xi implies p-th variation}, i.e., the construction of a reference path with linear $p$-th variation, to the $q$-adic setting; this is carried out in Theorem~\ref{thm:q-adic-linear-pvar} below. Once this is available, the later subsections extend naturally with only minor modifications, since their arguments rely mainly on the existence of a reference path with linear $p$-th variation along the chosen partition sequence, together with multiplicative transport and stability properties that are not specific to the dyadic case.

The following condition \eqref{con: uniform mag q-adic} is the natural $q$-adic analogue of Assumption~\ref{assum: uniform mag}, the uniform magnitude condition from Section~\ref{subsec: paths with linear p-th variation}. Namely, we require that at each level $m$, the dependence of the generalized Schauder coefficients on the spatial index $k$ is suppressed, while their dependence on the branch index $\ell$ is given by a fixed vector $a=(a_1,\dots,a_{q-1})\in\mathbb R^{q-1}$. Thus, the size at level $m$ is governed by a single scalar sequence $(c_m)_{m\ge0}$, and the relative weights of the $q-1$ generalized Schauder functions inside each parent interval are kept fixed across all levels and locations.

For a nonnegative sequence $(c_m)_{m\ge0}$, define coefficients
\begin{equation}    \label{con: uniform mag q-adic}
    \theta_{m,k,\ell}:=c_m a_\ell,
    \qquad m\ge0,\quad k=0,\dots,q^m-1,\quad \ell=1,\dots,q-1,
\end{equation}
and consider the series
\begin{equation}\label{eq:q-adic-series}
    x(t):=\sum_{m=0}^{\infty}\sum_{k=0}^{q^m-1}\sum_{\ell=1}^{q-1}\theta_{m,k,\ell}\,e_{m,k,\ell}(t),
    \qquad t\in[0,1].
\end{equation}
This assumption \eqref{con: uniform mag q-adic} is stronger than merely requiring a uniform magnitude condition at each level, but it is precisely what allows the contribution from level $m$ to be expressed through a single digit-dependent quantity $\eta_{d}(a)$ in the increment formula below.

Finally, define
\begin{equation}    \label{def: y_m q-adic}
    \xi_m^{(p,q)} := q^{m(1-\frac p2)}c_m^p, \qquad y_m:=q^{m(\frac{1}{p}-\frac{1}{2})}c_m=(\xi_m^{(p,q)})^{\frac 1p},
\end{equation}
and
\begin{equation}    \label{def: rho_q}
    \rho_q := q^{-(1-\frac 1p)}.
\end{equation}

We now present the $q$-adic analogue of Theorem~\ref{thm: xi implies p-th variation}.

\begin{theorem} \label{thm:q-adic-linear-pvar}
    Fix $p>1$, an integer $q \ge 2$, and $a\in\mathbb R^{q-1}\setminus\{0\}$. Assume that the sequence $(\xi_m^{(p,q)})_{m\ge0}$ converges as $m\to\infty$. Then the series \eqref{eq:q-adic-series} converges uniformly on $[0,1]$ to a continuous function $x$. Moreover, $x$ has linear $p$-th variation along the $q$-adic partition sequence $\T_q$, namely
    \[
        [x]^{(p)}_{\T_q}(t)
        =
        C_{p,q,a}\,t\,\lim_{m\to\infty}\xi_m^{(p,q)},
        \qquad t\in[0,1],
    \]
    where
    \[
        C_{p,q,a} := \mathbb E\Bigg[ \bigg| \sum_{j=1}^{\infty}\rho_q^j\,\eta_{D_j}(a) \bigg|^p \Bigg] \in(0, \infty),
    \]
    and $(D_j)_{j\ge1}$ are i.i.d.\ random variables uniformly distributed on $\{0,\dots,q-1\}$.
\end{theorem}

\medskip

\subsection{Time-changes and \texorpdfstring{$q$}{TEXT}-refining dense partition sequences}
\label{subsec:time-change}

We now use time-changes to transport the $q$-adic partitions to more general partitions. For the present purpose, it is enough to work with increasing homeomorphisms of $[0,1]$.

For a general partition sequence $\pi=(\pi^n)_{n\ge0}$, where
\[
    \pi^n=\{0=t_0^n<\cdots<t_{N(\pi^n)}^n=1\},
\]
we define, in analogy with Definition~\ref{Def: p-th variation},
\[
    [x]_{\pi^n}^{(p)}(t) := \sum_{i=0}^{N(\pi^n)-1}\big|x(t_{i+1}^n\wedge t)-x(t_i^n\wedge t)\big|^p,\qquad t\in[0,1].
\]
If $[x]_{\pi^n}^{(p)}(t)$ converges for every $t\in[0,1]$, we denote the limit by $[x]_\pi^{(p)}(t)$ and write $x\in V_\pi^p$.

The next proposition provides the basic time-change principle. We pull back the partition points by $\phi^{-1}$, and correspondingly push forward the path by composition with $\phi$.

\begin{lemma}   \label{lem:time-change-general}
    Let $\pi=(\pi^n)_{n\ge0}$ be a partition sequence on $[0,1]$, and let $\phi:[0,1]\to[0,1]$ be an increasing homeomorphism. Define
    \[
        \pi^n(\phi):=\phi^{-1}(\pi^n)
        =\{\phi^{-1}(t_i^n):\, i=0,\dots,N(\pi^n)\},\qquad n\ge0.
    \]
    If $x\in V_\pi^p$, then $x\circ\phi\in V_{\pi(\phi)}^p$ and $[x\circ\phi]_{\pi(\phi)}^{(p)}(t)=[x]_\pi^{(p)}(\phi(t))$ for $t\in[0,1]$. Moreover, if $x\in C^\alpha([0,1])$ and $\phi\in C^\beta([0,1])$ for some $\alpha,\beta\in(0,1]$, then $x\circ\phi\in C^{\alpha\beta}([0,1])$.
\end{lemma}

\begin{proof}
For $\pi^n=\{0=t_0^n<\cdots<t_{N(\pi^n)}^n=1\}$, set $s_i^n:=\phi^{-1}(t_i^n)$. Since $\phi$ is increasing, for every $t\in[0,1]$,
\[
    \phi(s_i^n\wedge t)=\phi(s_i^n)\wedge \phi(t)=t_i^n\wedge \phi(t).
\]
Hence
\begin{align*}
    [x\circ\phi]_{\pi^n(\phi)}^{(p)}(t)
    &=\sum_{i=0}^{N(\pi^n)-1}\big|x(\phi(s_{i+1}^n\wedge t))-x(\phi(s_i^n\wedge t))\big|^p
    \\
    &=\sum_{i=0}^{N(\pi^n)-1}\big|x(t_{i+1}^n\wedge \phi(t))-x(t_i^n\wedge \phi(t))\big|^p = [x]_{\pi^n}^{(p)}(\phi(t)).
\end{align*}
Taking $n\to\infty$ yields the first claim. The H\"older statement is immediate from
\[
    |x(\phi(t))-x(\phi(s))|
    \le \|x\|_{C^\alpha} |\phi(t)-\phi(s)|^\alpha
    \le \|x\|_{C^\alpha}\|\phi\|_{C^\beta}^\alpha |t-s|^{\alpha\beta}. \qedhere
\]
\end{proof}

We next specialize this time-change principle to partition sequences that are $q$-refining.

\begin{definition}  \label{Def: q-adically refining}
    Fix an integer $q\ge2$. A partition sequence $\P=(\P^n)_{n\ge0}$ on $[0,1]$ is called \emph{$q$-refining} if, for each $n\ge0$,
    \[
        \P^n=\{0=t_0^n<t_1^n<\cdots<t_{q^n}^n=1\}, \qquad \text{and} \qquad
        t_i^n=t_{qi}^{n+1} \quad \text{for all } i = 0, \cdots, q^n.
    \]
    We call $\P$ \emph{dense} if $\bigcup_{n\ge0}\P^n$ is dense in $[0,1]$.
\end{definition}
Note that any $q$-adic partition $\T_q$ in \eqref{def: q-adic partition} is $q$-refining and dense. The next result shows that there is a unique time-change (homeomorphism) between any dense $q$-refining partition sequence and the $q$-adic partition sequence $\T_q$.

\begin{proposition} \label{prop:q-refining-homeomorphism}
    Let $q\ge2$ and let $\P=(\P^n)_{n\ge0}$ be a dense $q$-refining partition sequence. Then there exists a unique increasing homeomorphism $\phi:[0,1]\to[0,1]$ such that
    \[
        \phi(t_i^n)=\frac{i}{q^n},\qquad i=0,\dots,q^n,\ \ n\ge0.
    \]
    Equivalently, $\P^n=\phi^{-1}(\T_q^n)$, for $n\ge0$.
\end{proposition}

\begin{proof}
Set $D:=\bigcup_{n\ge0}\P^n$ and define $\phi_0:D\to[0,1]$ by $\phi_0(t_i^n):=i/q^n$.
The $q$-adic refining property makes $\phi_0$ well defined, and $\phi_0$ is strictly increasing on $D$. Moreover,
    \[
    \phi_0(D)=\bigcup_{n\ge0}\T_q^n,
    \]
    which is dense in $[0,1]$.
    
    Define
    \[
    \phi(t):=\sup\{\phi_0(s):\, s\in D,\ s\le t\},\qquad t\in[0,1].
    \]
    Then $\phi$ is non-decreasing and extends $\phi_0$. If $s<t$, choose $u,v\in D$ such that $s<u<v<t$, which is possible because $D$ is dense. Then
    \[
    \phi(s)\le \phi(u)<\phi(v)\le \phi(t),
    \]
    so $\phi$ is strictly increasing.
    
Since every monotone function can only have jump discontinuities, suppose that $\phi$ is discontinuous at some $t\in[0,1]$. Then, at least one of the intervals
    \[
        \big(\phi(t-),\phi(t)\big) \qquad \text{or} \qquad \big(\phi(t),\phi(t+)\big)
    \]
    is nonempty (with the obvious one-sided interpretation at the endpoints). In either case, that interval is disjoint from $\phi([0,1])$. Thus, there exists a nonempty open interval $J\subset[0,1]$ such that $J\cap \phi([0,1])=\varnothing$. However, $\phi$ extends $\phi_0$, and
    \[
        \phi_0(D)=\bigcup_{n\ge0}\T_q^n
    \]
    is dense in $[0,1]$, so every nonempty open interval of $[0,1]$ must meet $\phi_0(D)\subset \phi([0,1])$, a contradiction. Therefore $\phi$ is continuous.
    
    Thus $\phi$ is a continuous strictly increasing map from $[0,1]$ onto $[0,1]$, hence an increasing homeomorphism. Uniqueness follows from continuity and the density of $D$. The identity $\phi(t_i^n)=i/q^n$ gives
    \[
        t_i^n=\phi^{-1}\Big(\frac{i}{q^n}\Big),
    \]
    and therefore $\P^n=\phi^{-1}(\T_q^n)$.
\end{proof}

\begin{remark}
If the dense $q$-refining partition sequence $\mathbb{P}$ is balanced (see Definition~2.1 of \cite{das2020}), the associated increasing homeomorphism $\phi$ and its inverse $\phi^{-1}$ from Proposition \ref{prop:q-refining-homeomorphism} are Lipschitz. To see this, we first restrict ourselves on $\cup_n \P^n$, as other points can be shown using the dense property of $\P$. Now take any $s,t\in \cup_n \P^n$, then there exists a large enough $n$ such that $t= t^n_\ell$ and $s=t^n_k$ for some $\ell$ and $k$. Hence, there exists some $c>0$ such that
\begin{align*}
|\phi(t)- \phi(s)| &=  |\phi(t^n_\ell)- \phi(t^n_k)| = |\frac{\ell}{q^n}- \frac{k}{q^n}| =\frac{1}{N(\P^n)} |\ell-k|\\&\leq  \big(\max_i |t^n_{i+1}-t^n_i|\big)|\ell-k| \leq c \big(\min_i |t^n_{i+1}-t^n_i|\big)|\ell-k| \leq  c|t^n_\ell - t^n_k| = c|t-s|.
\end{align*}
Here, the second inequality follows from the balanced property, and $N(\P^n)$ denotes the number of partition points of $\P^n$. Since $\T_q$ is also a balanced partition sequence, with the same line of argument one also has $\phi^{-1}$ Lipschitz.
\end{remark}

Combining Lemma~\ref{lem:time-change-general} and Proposition~\ref{prop:q-refining-homeomorphism}, we can transport the $q$-adic constructions to every dense $q$-refining partition sequence.

\begin{corollary}\label{cor:time-change-transport}
    Let $q\ge2$, let $\P$ be a dense $q$-refining partition sequence, and let $\phi$ be the associated homeomorphism from Proposition~\ref{prop:q-refining-homeomorphism}. If $x\in V_{\T_q}^p$ satisfies
    \[
        [x]_{\T_q}^{(p)}(t)=h(t),\qquad t\in[0,1],
    \]
    for some continuous non-decreasing function $h$ with $h(0)=0$, then $y:=x\circ\phi\in V_\P^p$ and
    \[
        [y]_\P^{(p)}(t)=h(\phi(t)),\qquad t\in[0,1].
    \]
    Equivalently, if $H:[0,1]\to[0,\infty)$ is continuous, non-decreasing, and $x\in V_{\T_q}^p$ satisfies
    \[
        [x]_{\T_q}^{(p)} \equiv H\circ\phi^{-1},
    \]
    then $x\circ\phi\in V_\P^p$ and $[x\circ\phi]_\P^{(p)} \equiv H$.
\end{corollary}

\begin{proof}
    Apply Lemma~\ref{lem:time-change-general} with $\pi=\T_q$ and use Proposition~\ref{prop:q-refining-homeomorphism}.
\end{proof}

The previous corollary shows that the problem of constructing a path with prescribed $p$-th variation along $\P$ reduces, via pullback by $\phi^{-1}$, to the corresponding prescribed-variation problem along $\T_q$. Applying the $q$-adic analogue of Theorem~\ref{thm: recipe} to the pulled-back target
\[
    h:=H\circ\phi^{-1},
\]
we therefore obtain the following constructive existence result for every dense $q$-refining partition sequence.

\begin{theorem}\label{thm:q-refining-recipe}
    Let $q\ge2$, let $\P$ be a dense $q$-refining partition sequence, and let $\phi:[0,1]\to[0,1]$ be the associated increasing homeomorphism from Proposition~\ref{prop:q-refining-homeomorphism}, so that
    \[
        \P^n=\phi^{-1}(\T_q^n),\qquad n\ge0.
    \]
    Let $H:[0,1]\to[0,\infty)$ be continuous and non-decreasing with $H(0)=0$, and set
    \[
        h:=H\circ\phi^{-1}.
    \]
    Suppose that $h\in C^1([0,1])$ and that $\big[(h')^{1/p}\big]_{\T_q}^{(p)}\equiv 0$ for some $p>1$. Then there exists a path $y\in V_\P^p$ such that
    \[
        [y]_\P^{(p)}(t)=H(t),\qquad t\in[0,1].
    \]
    Moreover, if $(h')^{1/p}\in C^\alpha([0,1])$ for some $\alpha \in (0, 1/p]$, and if
    $\phi\in C^\beta([0,1])$ for some $\beta\in(0,1]$, then the constructed path
    $y$ belongs to $C^{\alpha\beta}([0,1])\cap V_\P^p$.
\end{theorem}

\begin{proof}
    By applying along $\T_q$ the same multiplicative construction as in the proof of Theorem~\ref{thm: recipe}, using Theorem~\ref{thm:q-adic-linear-pvar} to provide a reference path with linear $p$-th variation, one obtains $x\in V_{\T_q}^p$ such that
    \[
        [x]_{\T_q}^{(p)} \equiv h.
    \]
    Set $y:=x\circ\phi$. Then Corollary~\ref{cor:time-change-transport} gives
    \[
    [y]_\P^{(p)}(t)
    =
    [x]_{\T_q}^{(p)}(\phi(t))
    =
    h(\phi(t))
    =
    H(t),
    \qquad t\in[0,1].
    \]
    If in addition $(h')^{1/p}\in C^\alpha([0,1])$ for some $\alpha\le 1/p$, then the same $q$-adic theorem yields $x\in C^\alpha([0,1])\cap V_{\T_q}^p$. Therefore, if also $\phi\in C^\beta([0,1])$ for some $\beta\in(0,1]$,
    Lemma~\ref{lem:time-change-general} implies
    \[
        y=x\circ\phi\in C^{\alpha\beta}([0,1])\cap V_\P^p.
    \]
\end{proof}

\begin{remark}
    Suppose in addition that $\phi$ is a $C^{1+\gamma}$-diffeomorphism of $[0,1]$ for some $\gamma>1/p$. Then the hypothesis on $h=H\circ\phi^{-1}$ in Theorem~\ref{thm:q-refining-recipe} can be replaced by the more direct condition
    \[
        H\in C^1([0,1]), \qquad \big[(H')^{1/p}\big]_{\P}^{(p)}\equiv 0.
    \]
    Indeed, if $(H')^{1/p}$ has vanishing $p$-th variation along $\P$, then by Corollary~\ref{cor:time-change-transport},
    \[
    [(H'\circ \phi^{-1})^{1/p}]^{(p)}_{\T_q} \equiv 0.
    \]
    Moreover, since $(\phi^{-1})'\in C^\gamma([0,1])$ and $(\phi^{-1})'>0$ on $[0,1]$, compactness implies that $(\phi^{-1})'$ is bounded away from zero: there exists $c>0$ such that $(\phi^{-1})'(t)\ge c$ for all $t\in[0,1]$. The map $x\mapsto x^{1/p}$ is $C^1$ on $[c,\infty)$, hence Lipschitz there. Therefore, $((\phi^{-1})')^{1/p}\in C^\gamma([0,1])$ as well and Lemma~\ref{lem:inclusions} (iii) implies that $((\phi^{-1})')^{1/p}$ also has vanishing $p$-th variation along $\T_q$. Since
    \[
        (h')^{1/p} = (H' \circ \phi^{-1})^{1/p} \, \big((\phi^{-1})'\big)^{1/p},
    \]
    it follows from the $q$-adic version of Proposition~\ref{prop: translation} that
    \[
        \big[(h')^{1/p}\big]_{\T_q}^{(p)}\equiv 0.
    \]
    Therefore Theorem~\ref{thm:q-refining-recipe} applies.
\end{remark}

We conclude with the observation that, since $\phi$ is a homeomorphism of $[0,1]$, the composition operator
\[
    C_\phi:C([0,1])\to C([0,1]),\qquad C_\phi(x):=x\circ\phi,
\]
is a linear isometric isomorphism on $C([0,1])$ equipped with the uniform norm; indeed,
\[
    \|x\circ\phi\|_\infty=\|x\|_\infty,
    \qquad
    C_\phi^{-1}=C_{\phi^{-1}}.
\]
Therefore, density statements (Section~\ref{subsec: dense}) transfer immediately from the $q$-adic setting to a dense $q$-refining partition sequence $\P$. Likewise, Banach subspaces and stability statements (Sections~\ref{subsec: Banach subspaces} and \ref{subsec: stability Ito}) obtained in the $q$-adic setting can be transported to $\P$ by pulling back the norm through $C_\phi$.

\bigskip

\section{Proofs}    \label{sec: proofs}

\subsection{Proof of Theorem~\ref{thm: xi implies p-th variation}}    \label{subsec: proof main}

To prove Theorem~\ref{thm: xi implies p-th variation}, we first introduce some notations and preliminary lemmas.

For $n\ge1$ and $k\in I_n$, write the level-$n$ dyadic interval
\[
    I_{n,k}:=\Big[\frac{k}{2^n},\frac{k+1}{2^n}\Big).
\]
For $0\le m\le n-1$, there is a unique index
\[
    \kappa(m;n,k):=\Big\lfloor \frac{k}{2^{n-m}}\Big\rfloor\in I_m
\]
such that $I_{n,k}\subset\big[\kappa(m;n,k)/2^m,\ (\kappa(m;n,k)+1)/2^m\big)$. Since the Haar function $\psi_{m,\kappa(m;n,k)}$ of \eqref{Eq:Haar} is constant on $I_{n,k}$, we may define its \emph{Haar sign} on $I_{n,k}$ by
\begin{equation}    \label{eq:haar-sign-on-Ink}
    \varepsilon_{m}(n,k)\in\{\pm1\}
    \quad\text{via}\quad
    \psi_{m,\kappa(m;n,k)}(t)=\varepsilon_{m}(n,k)\,2^{m/2},
    \qquad \forall\,t\in I_{n,k}.
\end{equation}
Under Assumption~\ref{assum: uniform mag}, define the \emph{level signs} for $j=1,\dots,n$ by
\begin{equation}    \label{eq:level-signs-epsj}
    \varepsilon_j(k) := \sigma_{n-j,\kappa(n-j;n,k)} \,\varepsilon_{n-j}(n,k)\in\{\pm1\}, \qquad j=1,\dots,n,\ \ k\in I_n.
\end{equation}
Roughly speaking, for a fixed finest interval $I_{n,k}$, the quantity $\varepsilon_j(k)$ records the \emph{effective sign} of the level-$(n-j)$ contribution to the increment of $x$ over $I_{n,k}$. This becomes clear from the following calculation:
\begin{align*}
    x(\frac{k+1}{2^n}) - x(\frac{k}{2^n}) &= \sum_{m=0}^{n-1} \theta_{m, \kappa(m;n,k)} \Big( e_{m, \kappa(m;n,k)}(\frac{k+1}{2^n}) - e_{m, \kappa(m;n,k)}(\frac{k}{2^n}) \Big)
    \\
    &= \sum_{m=0}^{n-1} c_m \sigma_{m, \kappa(m;n,k)} \Big( \varepsilon_m(n, k) 2^{m/2}2^{-n} \Big) = \sum_{m=0}^{n-1}  c_m 2^{m/2}2^{-n} \big( \sigma_{m, \kappa(m;n,k)} \varepsilon_m(n, k) \big)
    \\
    &= \sum_{j = 1}^{n} c_{n-j} 2^{(n-j)/2}2^{-n} \Big( \sigma_{n-j, \kappa(n-j;n,k)} \varepsilon_{n-j}(n, k) \Big)
    = \sum_{j = 1}^{n} c_{n-j} 2^{(n-j)/2}2^{-n} \Big( \varepsilon_j(k) \Big)
\end{align*}
Indeed, $\psi_{n-j,\kappa(n-j;n,k)}$ is constant on $I_{n,k}$ with Haar sign $\varepsilon_{n-j}(n,k)$, while under
Assumption~\ref{assum: uniform mag} the corresponding Faber-Schauder coefficient carries an additional coefficient sign
$\sigma_{n-j,\kappa(n-j;n,k)}$. Their product $\varepsilon_j(k)$ is therefore the sign with which the level-$(n-j)$ ``tent'' contributes to the dyadic difference $x(\frac{k+1}{2^n})-x(\frac{k}{2^n})$ in the following lemma.

\begin{lemma}       \label{lem:increment-rademacher}
For $n\ge1$ and $k\in I_n$, set $\Delta_k^n x:=x(\frac{k+1}{2^n})-x(\frac{k+1}{2^n})$. Under Assumption~\ref{assum: uniform mag}, we define for any $p>1$
    \begin{equation}\label{eq:def-y-rho}
        y_m:=2^{m(\frac1p-\frac12)}c_m
        \qquad\text{and}\qquad
        \rho:=2^{-(1-\frac1p)}\in(0,1).
    \end{equation}
    Then, we have the identity
    \begin{equation}\label{eq:dyadic-increment-rademacher}
        2^{\frac{n}{p}}\Delta_k^n x = \sum_{j=1}^{n}\rho^{j}\,y_{n-j}\,\varepsilon_j(k),
    \end{equation}
    with the level sign notation $\varepsilon_j(k)$ of \eqref{eq:level-signs-epsj}. Moreover, we have the relationship for every $m \ge 0$
    \begin{equation}\label{eq:xi-y-connection}
        \xi_m^{(p)} = 2^{-\frac{mp}{2}}\sum_{k\in I_m}|\theta_{m,k}|^p = 2^{m(1-\frac{p}{2})}c_m^p, \qquad\text{hence}\qquad y_m=\big(\xi_m^{(p)}\big)^{1/p}.
    \end{equation}
\end{lemma}

\begin{proof}
    Consider the $n$-th Faber-Schauder partial sum for $n \in \mathbb{N}$
    \[
        x_n(t) := \sum_{m=0}^{n-1}\sum_{\ell\in I_m}\theta_{m,\ell}\,e_{m,\ell}(t).
    \]
    Since $x_n(t_k^n)=x(t_k^n)$ for all $k\in\{0,\dots,2^n\}$, $\Delta_k^n x=\Delta_k^n x_n$ holds. Since $e_{m,\ell}'(t)=\psi_{m,\ell}(t)$ holds for almost every $t$ and $x_n$ is affine on each $I_{n,k}$, we have for a.e. $t\in I_{n,k}$,
    \[
        (x_n)'(t) = \sum_{m=0}^{n-1} \theta_{m,\kappa(m;n,k)}\,\psi_{m,\kappa(m;n,k)}(t),
    \]
    and integrating over $I_{n,k}$ gives
    \[
        \Delta_k^n x = 2^{-n}\sum_{m=0}^{n-1}\theta_{m,\kappa(m;n,k)}\,\psi_{m,\kappa(m;n,k)}(t)
        \qquad \text{for any fixed } t\in I_{n,k}.
    \]
    Using \eqref{eq:haar-sign-on-Ink} and Assumption~\ref{assum: uniform mag} yields 
    \begin{equation*}
        \Delta_k^n x
        = 2^{-n}\sum_{m=0}^{n-1} 2^{m/2}\,c_m\,\varepsilon_{n-m}(k).
    \end{equation*}
    Hence, the identity \eqref{eq:dyadic-increment-rademacher} follows from \eqref{eq:def-y-rho} by reindexing $j=n-m$. Finally, \eqref{eq:xi-y-connection} is immediate from \eqref{eq:TL-signed-coeff} and $|I_m|=2^m$.
\end{proof}

The following Lemma~\ref{lem:sign-matrix} is from \cite[Proof of Theorem~2.1]{Misura2019}. The key observation is that, as $k$ runs over $I_n$, the vector of level signs $\big(\varepsilon_1(k),\dots,\varepsilon_n(k)\big)$ exhausts \emph{all} sign patterns in $\{\pm1\}^n$ exactly once. Equivalently, the $2^n\times n$ sign matrix whose $k$-th row is $\big(\varepsilon_1(k),\dots,\varepsilon_n(k)\big)$ is a permutation of the standard Rademacher sign matrix. Consequently, any average over dyadic intervals at level $n$ of a function of $\big(\varepsilon_1(k),\dots,\varepsilon_n(k)\big)$ coincides with the expectation of the same function applied to i.i.d.\ Rademacher variables, as stated in \eqref{eq:average-equals-expectation}.

\begin{lemma}   \label{lem:sign-matrix}
    Fix $n \in \mathbb{N}$. The map
    \[
        I_n\ni k\longmapsto \big(\varepsilon_1(k),\dots,\varepsilon_n(k)\big)\in\{\pm1\}^n
    \]
    is a bijection. Consequently, for any function $\Phi:\{\pm1\}^n\to\R$,
    \begin{equation}\label{eq:average-equals-expectation}
        2^{-n}\sum_{k\in I_n}\Phi\big(\varepsilon_1(k),\dots,\varepsilon_n(k)\big) =\E\Big[\Phi(\varepsilon_1,\dots,\varepsilon_n)\Big],
    \end{equation}
    where $\varepsilon_1,\dots,\varepsilon_n$ are i.i.d.\ Rademacher random variables.
\end{lemma}

We are now ready to prove Theorem~\ref{thm: xi implies p-th variation}. The proof is broken into 4 steps for convenience of readers.

\begin{proof}[Proof of Theorem~\ref{thm: xi implies p-th variation}]
    \noindent\textbf{Step }1: $x$ is $\frac1p$-H\"older continuous.

    From the identity \eqref{eq:xi-y-connection}, the existence of the limit of $\xi^{(p)}_m$ implies
    \[
        \sup_{m \ge 0} \big( 2^{m(1-\frac p2)} c_m^p \big) < \infty \qquad \iff \qquad \sup_{m \ge 0} \big( 2^{m(\frac 1p-\frac 12)} c_m \big) < \infty.
    \]
    Since $|\theta_{m, k}| = c_m$ for all $k \in I_m$ by Assumption~\ref{assum: uniform mag}, this is also equivalent to condition \eqref{con: Holder} of Lemma~\ref{lem: Ciesielski} when setting $\alpha = \frac 1p$. Hence, the resulting function $x$ is $\frac1p$-H\"older continuous.

    \medskip
    
    \noindent\textbf{Step 2}: The identity \eqref{eq: xi implies p-th variation} holds for terminal time, i.e. $t=1$.
    
    By Lemma~\ref{lem:increment-rademacher} and Lemma~\ref{lem:sign-matrix}, we have for each $n\in\N$,
    \begin{equation}\label{eq:Vn-expectation}
        [x]^{(p)}_{\T^n}(1) = \sum_{k\in I_n}|\Delta^n_k x|^p
        = 2^{-n}\sum_{k\in I_n}\Big|\sum_{j=1}^n \rho^j y_{n-j}\,\varepsilon_j(k)\Big|^p
        = \E\Big[\Big|\sum_{j=1}^n \rho^j y_{n-j}\,\varepsilon_j\Big|^p\Big],
    \end{equation}
    where $(\varepsilon_j)_{j\ge1}$ are i.i.d.\ Rademacher random variables.
    
    Since $\xi^{(p)}_m=y_m^p$ converges as $m\to\infty$, there exists $y\ge0$ such that $y_m \xrightarrow{m\to\infty}y$. Set
    \[
        Z:=\sum_{j=1}^{\infty}\rho^j \varepsilon_j, \qquad C_p:=\E[|Z|^p]\in(0,\infty),
    \]
    where the series $Z$ converges a.s.\ and in $L^p$ since $\sum_{j\ge1}\rho^j<\infty$ and $p>1$. We now claim the convergence
    \begin{equation}    \label{eq:Sp-to-yZ}
        \sum_{j=1}^n \rho^j y_{n-j}\varepsilon_j \;\xrightarrow[n\to\infty]{\;\;L^p\;\;} y\sum_{j=1}^{\infty}\rho^j\varepsilon_j = yZ.
    \end{equation}
    Indeed, fix $J\in\N$.  Write
    \[
        \sum_{j=1}^n \rho^j y_{n-j}\varepsilon_j = \sum_{j=1}^J \rho^j y_{n-j}\varepsilon_j +\sum_{j=J+1}^n \rho^j y_{n-j}\varepsilon_j.
    \]
    Since $y_{n-j} \xrightarrow{n\to\infty} y$ for each fixed $j\le J$, the first sum converges to $\sum_{j=1}^J\rho^j y\varepsilon_j$ in $L^p$. For the tail part, using $\sup_m y_m<\infty$ (which follows from $y_m\to y$) we obtain
    \[
        \Big\|\sum_{j=J+1}^n \rho^j y_{n-j}\varepsilon_j\Big\|_{L^p} \le \sum_{j=J+1}^{\infty}\rho^j \sup_m y_m \;\xrightarrow{J\to\infty}\;0,
    \]
    uniformly in $n\ge J+1$.  This yields the claimed $L^p$ convergence \eqref{eq:Sp-to-yZ}.
    
    Consequently, by \eqref{eq:Vn-expectation} and convergence in $L^p$,
    \begin{equation}    \label{eq:Vn-limit}
        [x]^{(p)}_{\T^n}(1) = \E\Big[\Big|\sum_{j=1}^n \rho^j y_{n-j}\varepsilon_j\Big|^p\Big] \;\xrightarrow{n\to\infty} \;\E[|yZ|^p] = y^p C_p = C_p \lim_{n\to\infty} \xi^{(p)}_n.
    \end{equation}

    \medskip
    \noindent\textbf{Step 3}: The identity \eqref{eq: xi implies p-th variation} holds for dyadic points $t$.
    
    Fix $m\in\N$ and a dyadic point $t=\ell 2^{-m}$ with $\ell\in\{0,1,\dots,2^m\}$.
    For each $n\ge m$, decompose $\{0,1,\dots,2^n-1\}$ into $2^m$ consecutive blocks of length $2^{n-m}$:
    \[
        B_{n,m}(r):=\{r2^{n-m},\dots,(r+1)2^{n-m}-1\},\qquad r=0,\dots,2^m-1.
    \]
    Then
    \[
        [x]^{(p)}_{\T^n}(t)
        =\sum_{k=0}^{\ell 2^{n-m}-1}|\Delta_k^n x|^p
        =\sum_{r=0}^{\ell-1}\ \sum_{k\in B_{n,m}(r)}|\Delta_k^n x|^p.
    \]
    Set
    \[
        V_{n,m}(r):=\sum_{k\in B_{n,m}(r)}|\Delta_k^n x|^p,\qquad r=0,\dots,2^m-1,
    \]
    so that $[x]^{(p)}_{\T^n}(1)=\sum_{r=0}^{2^m-1}V_{n,m}(r)$ and
    $[x]^{(p)}_{\T^n}(t)=\sum_{r=0}^{\ell-1}V_{n,m}(r)$.

    We claim that for each fixed $m$,
    \begin{equation}\label{eq:block-equipartition}
        \max_{0\le r\le 2^m-1}\Big|V_{n,m}(r)-2^{-m}[x]^{(p)}_{\T^n}(1)\Big|
        \xrightarrow[n\to\infty]{}0 .
    \end{equation}
    Indeed, using \eqref{eq:dyadic-increment-rademacher}, write for $k\in I_n$
    \[
        2^{\frac np}\Delta_k^n x=\sum_{j=1}^n \rho^j y_{n-j}\,\varepsilon_j(k).
    \]
    For $k\in B_{n,m}(r)$, decompose the sum into the head and tail parts
    \[
        U_{n,m}(k):=\sum_{j=1}^{n-m}\rho^j y_{n-j}\,\varepsilon_j(k),
        \qquad
        T_{n,m}(r):=\sum_{j=n-m+1}^{n}\rho^j y_{n-j}\,\varepsilon_j(r2^{n-m}),
    \]
    so that
    \[
        2^{\frac np}\Delta_k^n x = \sum_{j=1}^n \rho^j y_{n-j}\,\varepsilon_j(k)=U_{n,m}(k)+T_{n,m}(r),
        \qquad k\in B_{n,m}(r).
    \]
    Here, note that for $k\in B_{n,m}(r)$ and $j\in\{n-m+1,\dots,n\}$ the sign $\varepsilon_j(k)$ depends only on $\lfloor k/2^{n-m}\rfloor=r$, hence it is constant over the block. Moreover, since $Y:=\sup_{m\ge0}y_m<\infty$ (as $y_m \xrightarrow{m\to\infty} y$), we have the uniform tail bound
    \[
        \max_{0\le r\le 2^m-1}|T_{n,m}(r)|
        \le Y\sum_{j=n-m+1}^{\infty}\rho^j
        \le \frac{Y \rho}{1-\rho} \,\rho^{\,n-m}.
    \]
    Also, $\sup_{n,k}|U_{n,m}(k)|\le Y\sum_{j\ge1}\rho^j=:M<\infty$, hence the map
    $u\mapsto |u|^p$ is Lipschitz on $[-2M,2M]$ and there exists $L=L(p,M)$ such that
    \[
        \big||u+v|^p-|u|^p\big|\le L|v|,\qquad |u|\le M,\ |v|\le M.
    \]
    Therefore, for each $r$,
    \begin{align*}
        \Big|V_{n,m}(r)-2^{-n}\sum_{k\in B_{n,m}(r)}|U_{n,m}(k)|^p\Big|
        &=2^{-n}\Big|\sum_{k\in B_{n,m}(r)}\big(|U_{n,m}(k)+T_{n,m}(r)|^p-|U_{n,m}(k)|^p\big)\Big|\\
        &\le 2^{-n}\sum_{k\in B_{n,m}(r)}L|T_{n,m}(r)|
        \\
        &\le L\,2^{-m}\max_{0\le r \le 2^m-1}|T_{n,m}(r)|
        \le \frac{LY \rho}{1-\rho}\,2^{-m}\rho^{n-m}.
    \end{align*}
    By Lemma~\ref{lem:sign-matrix}, the multiset
    $\{(\varepsilon_1(k),\dots,\varepsilon_{n-m}(k)):\ k\in B_{n,m}(r)\}$
    is the same for every $r$, hence the quantity
    \[
        A_{n,m}:=2^{-n}\sum_{k\in B_{n,m}(r)}|U_{n,m}(k)|^p
    \]
    is independent of $r$. It follows that
    \[
        \max_{0\le r\le 2^m-1}|V_{n,m}(r)-A_{n,m}|\le\frac{LY \rho}{1-\rho}\,2^{-m}\rho^{n-m}.
    \]
    Since $\sum_{r=0}^{2^m-1}V_{n,m}(r)=[x]^{(p)}_{\T^n}(1)$, we also have $A_{n,m}=2^{-m}[x]^{(p)}_{\T^n}(1)+O(2^{-m}\rho^{n-m})$, and \eqref{eq:block-equipartition} follows.

    Consequently, for dyadic $t=\ell2^{-m}$,
    \[
        [x]^{(p)}_{\T^n}(t)
        =\sum_{r=0}^{\ell-1}V_{n,m}(r)
        =\frac{\ell}{2^m}[x]^{(p)}_{\T^n}(1) + o_{n\to\infty}(1),
    \]
    where the $o_{n\to\infty}(1)$ term depends on $m$, but is uniform over $\ell\in\{0,\dots,2^m\}$, and converges to zero as $n\to\infty$. Letting $n\to\infty$ and using \eqref{eq:Vn-limit}, we obtain
    \begin{equation}\label{eq:Vn-dyadic-limit}
        \lim_{n\to\infty}[x]^{(p)}_{\T^n}(t)
        = t \lim_{n\to\infty}[x]^{(p)}_{\T^n}(1) = C_p t \lim_{n\to\infty} \xi^{(p)}_n,
        \qquad t\in\{0,2^{-m},\dots,1\}.
    \end{equation}

    \medskip
    \noindent\textbf{Step 4}: The identity \eqref{eq: xi implies p-th variation} holds for all $t \in [0, 1]$.

    Let $t\in[0,1]$. For each $m\in\N$, define the dyadic approximations
    \[
        t^-_m:=\frac{\lfloor 2^m t\rfloor}{2^m},
        \qquad
        t^+_m:=\frac{\lfloor 2^m t\rfloor+1}{2^m},
    \]
    so that $t^-_m\le t\le t^+_m$ and $t^-_m,t^+_m \rightarrow t$ as $m \rightarrow \infty$.
    Since $[x]^{(p)}_{\T^n}(\cdot)$ is non-decreasing, for every $n\ge 0$ we have
    \[
        [x]^{(p)}_{\T^n}(t^-_m)\le [x]^{(p)}_{\T^n}(t)\le [x]^{(p)}_{\T^n}(t^+_m).
    \]
    Let $L:=\lim_{n\to\infty}[x]^{(p)}_{\T^n}(1)$, which exists by Step~2. 
    Fix $m\in\N$. By Step~3, we have
    \[
        \lim_{n\to\infty}[x]^{(p)}_{\T^n}(t^-_m)=t^-_m\,L,
        \qquad
        \lim_{n\to\infty}[x]^{(p)}_{\T^n}(t^+_m)=t^+_m\,L.
    \]
    Taking $\liminf_{n\to\infty}$ and $\limsup_{n\to\infty}$ in the monotone sandwich yields
    \[
        t^-_m\,L
        \le \liminf_{n\to\infty}[x]^{(p)}_{\T^n}(t)
        \le \limsup_{n\to\infty}[x]^{(p)}_{\T^n}(t)
        \le t^+_m\,L.
    \]
    Letting $m\to\infty$ gives $t^-_m\uparrow t$ and $t^+_m\downarrow t$, hence
    \[
        \liminf_{n\to\infty}[x]^{(p)}_{\T^n}(t)
        =\limsup_{n\to\infty}[x]^{(p)}_{\T^n}(t)
        = tL.
    \]
    Combining this with \eqref{eq:Vn-limit} concludes \eqref{eq: xi implies p-th variation} for all $t\in[0,1]$.
\end{proof}

\medskip

\subsection{Proof of Proposition \ref{prop: translation}}   \label{subsec: proof translation}
\begin{proof}[Proof of Proposition \ref{prop: translation}]
    Fix any $t\in [0, 1]$. Then
    \begin{align*}
        [y]^{(p)}_{\T^n}(t)
        &=\sum_{i=0}^{2^n-1}\Big|g(t_{i+1}^n\wedge t)x(t_{i+1}^n\wedge t)-g(t_i^n\wedge t)x(t_i^n\wedge t)\Big|^p \\
        &=\sum_{i=0}^{2^n-1}\Big|
            \underbrace{g(t_i^n\wedge t)\big(x(t_{i+1}^n\wedge t)-x(t_i^n\wedge t)\big)}_{A_i^n}
            +\underbrace{x(t_{i+1}^n\wedge t)\big(g(t_{i+1}^n\wedge t)-g(t_i^n\wedge t)\big)}_{B_i^n}
          \Big|^p .
    \end{align*}
    By Minkowski's inequality,
    \[
        \Big|\Big(\sum_{i}|A_i^n+B_i^n|^p\Big)^{1/p}-\Big(\sum_{i}|A_i^n|^p\Big)^{1/p}\Big|
        \le \Big(\sum_{i}|B_i^n|^p\Big)^{1/p}.
    \]
    We claim that $\big(\sum_{i}|B_i^n|^p\big)^{1/p}\to0$ as $n\to\infty$:
    \[
        \sum_{i=0}^{2^n-1}|B_i^n|^p
        \le \|x\|_\infty^p \sum_{i=0}^{2^n-1}\big|g(t_{i+1}^n\wedge t)-g(t_i^n\wedge t)\big|^p
        = \|x\|_\infty^p\,[g]^{(p)}_{\T^n}(t) \longrightarrow \|x\|_\infty^p\,[g]^{(p)}_{\T}(t)=0.
    \]
    Next, note that
    \[
        \sum_{i=0}^{2^n-1}|A_i^n|^p
        =\sum_{i=0}^{2^n-1} \big|g(t_i^n\wedge t)\big|^p\,\big|x(t_{i+1}^n\wedge t)-x(t_i^n\wedge t)\big|^p
        \le \|g\|_\infty^p\,[x]^{(p)}_{\T^n}(t),
    \]
    and $\sup_n [x]^{(p)}_{\T^n}(t)<\infty$ because $x\in V^p_{\T}$. Hence $\big(\sum_i|A_i^n|^p\big)^{1/p}$ is uniformly bounded in $n$, and so is $\big(\sum_i|A_i^n+B_i^n|^p\big)^{1/p}$ by the triangle inequality. Therefore, on a bounded interval the map $u\mapsto u^p$ is Lipschitz, and the convergence of $\ell^p$-norms implies
    \[
    \sum_{i}|A_i^n+B_i^n|^p-\sum_{i}|A_i^n|^p \xrightarrow{n\to\infty} 0.
    \]
    Consequently, the two limits coincide:
    \begin{equation}\label{eq:reduce-translation}
        \lim_{n\to\infty} [y]^{(p)}_{\T^n}(t)
        = \lim_{n\to\infty} \sum_{i=0}^{2^n-1} \big|g(t_i^n\wedge t)\big|^p\,\big|x(t_{i+1}^n\wedge t)-x(t_i^n\wedge t)\big|^p.
    \end{equation}
    
    Now define the discrete measures (see Definition~1.1, Lemma~1.3 of \cite{perkowski2019})
    \[
        \mu_n^x:=\sum_{i=0}^{2^n-1}\delta_{t_i^n}\, \big|x(t_{i+1}^n)-x(t_i^n) \big|^p.
    \]
    Since $x\in V^p_{\T}$, the measures $\mu_n^x$ converge weakly to a finite Borel measure $\mu^x$ without atoms, and
    \[
        [x]^{(p)}_{\T}(t)=\mu^x([0,t]),\qquad t\in[0,1].
    \]
    Fix $t\in[0,1]$ and consider the restriction $\mu_n^{x,t}:=\mu_n^x|_{[0,t]}$ (and similarly $\mu^{x,t} := \mu^x|_{[0, t]}$). Since $\mu^x$ has no atom at $t$, we still have $\mu_n^{x,t}\Rightarrow \mu^{x,t}$. Because $u\mapsto |g(u)|^p$ is continuous and bounded on $[0,1]$, weak convergence yields
    \[
        \int_{[0,t]} |g(u)|^p\,\mu_n^x(du)
        =\int_{[0,t]} |g(u)|^p\,\mu_n^{x,t}(du)
        \xrightarrow{n\to\infty}
        \int_{[0,t]} |g(u)|^p\,\mu^{x,t}(du)
        =\int_{0}^{t}|g(u)|^p\,d[x]^{(p)}_{\T}(u).
    \]
    Finally, the sum in \eqref{eq:reduce-translation} differs from $\int_{[0,t]} |g(u)|^p\,\mu_n^x(du)$ by at most one term (the interval straddling $t$), which vanishes as $n\to\infty$ by continuity of $x$ and boundedness of $g$. Combining with \eqref{eq:reduce-translation} gives
    \[
        \lim_{n\to\infty}[y]^{(p)}_{\T^n}(t)=\int_{0}^{t}|g(u)|^p\,d[x]^{(p)}_{\T}(u),
    \]
    which proves that $y\in V^p_{\T}$ and establishes the identity \eqref{eq: p-th variation of y}.
\end{proof}

\medskip

\subsection{Proof of Theorem~\ref{thm:densityoffnsinVp}}    \label{subsec: proof density}

\begin{proof}[Proof of Theorem \ref{thm:densityoffnsinVp}]
    The proof consists of two steps.
    
    \textbf{Step 1}: Given $x\in V^p_{\T} \cap C^\alpha([0,1])$ and $y\in V^{p,\tau}_\T\cap C^\alpha([0,1])$, we shall construct a sequence $\{y_n\}$ with prescribed $p$-th variation $\tau$ that converges uniformly to $x$. We consider their Faber-Schauder coefficients $\{\theta^x_{m, k}\}_{m \ge 0, k \in I_m}, \{\theta^y_{m, k}\}_{m \ge 0, k \in I_m}$, and construct a sequence of functions
    \[
        y_n(t) := x(0) + \big(x(1)-x(0)\big)t + \sum_{m = 0}^{n-1} \sum_{k \in I_m} \theta^x_{m,k} e_{m,k}(t) + \sum_{m = n}^\infty \sum_{k \in I_m} \theta^y_{m,k} e_{m,k}(t), \qquad \forall \, t \in [0, 1].
    \]
    In other words, the difference $y_n - y$ is a piecewise linear with breakpoints in $\T^n$. Since adding a piecewise linear function does not change the $p$-th variation \cite[Lemma~3.1]{Misura2019}, we have
    \begin{equation}    \label{eq: same p-th variation}
        [y_n]^{(p)}_\T = [y]^{(p)}_\T = \tau, \qquad \text{for every } n \in \N,    
    \end{equation}
    and by construction $x(k2^{-n}) = y_n(k2^{-n})$ for all $k = 0,\dots, 2^n$. Hence, the triangle inequality yields
    \begin{equation}    \label{ineq: tri}
        |x(t) - y_n(t)| \leq |x(t) - x(t_n)| + |y_n(t_n) - y_n(t)|, \qquad \forall \, t\in[0, 1],
    \end{equation}
    where $t_n$ is the nearest dyadic point before $t$ at level $n$, i.e., $t_n := \lfloor t 2^n \rfloor 2^{-n}$.
    
    On the other hand, since both $x,y$ are $\alpha$-H\"older continuous, Lemma \ref{lem: Ciesielski} implies 
    \[
        \sup_{m,k} \big( 2^{m(\alpha - \frac12)} |\theta^x_{m,k}| \big) < \infty, \qquad \sup_{m, k} \big( 2^{m(\alpha - \frac12)} |\theta^y_{m,k}| \big) < \infty.
    \]
    Another application of Lemma \ref{lem: Ciesielski} to each $y_n$ shows that every $y_n$ is also $\alpha$-H\"older continuous and $C_y := \sup_{n \in \N} \Vert y_n \Vert_{C^{\alpha}} < \infty$ since the Faber-Schauder coefficients of $y_n$ are uniformly bounded by those of $x$ and $y$. The inequality \eqref{ineq: tri} can be further bounded as 
    \[
        |x(t) - y_n(t)| \leq |x(t) - x(t_n)| + |y_n(t_n) - y_n(t)| \leq \big( \Vert x \Vert_{C^{\alpha}}+C_y \big) |t-t_n|^{\alpha} \le \big( \Vert x \Vert_{C^{\alpha}}+C_y \big) 2^{-n\alpha},
    \]
    for any $t \in [0, 1]$, thus
    \begin{equation}    \label{ineq: supremum norm bound}
        \Vert x - y_n \Vert_{\infty} \le \big( \Vert x \Vert_{C^{\alpha}}+C_y \big) 2^{-n\alpha}.
    \end{equation}
    
    \medskip

    \noindent \textbf{Step 2}: For any $z \in C([0, 1])$, there exists a sequence of Bernstein polynomials $\{z^n\}_{n \in \N}$ such that $\Vert z - z^n \Vert_{\infty} \to 0$ as $n \to \infty$ (by the Stone-Weierstrass theorem). Note that as a polynomial, each $z^n \in C^1([0, 1]) \cap V^p_{\T}$ for any $p > 1$, since $[z^n]^{(p)}_{\T} \equiv 0$. Suppose that we are given $y \in V^{p, \tau}_{\T} \cap C^{\alpha}([0, 1])$. Fix $z^n$ for some $n \in \N$ and apply \textbf{Step 1} argument to $z^n$ (as $x$) and $y$ so that we can construct a sequence of functions $\{y^n_m\}_{m \in \N}$ satisfying
    \begin{equation*}
        [y^n_m]^{(p)}_{\T} = [y]^{(p)}_{\T} \qquad \text{and} \qquad \Vert z^n-y^n_m \Vert_{\infty} \le \big( \Vert z^n \Vert_{C^{\alpha}}+C_y \big) 2^{-m\alpha}, \qquad \text{for all } m \in \N, 
    \end{equation*}
    by virtue of \eqref{eq: same p-th variation} and \eqref{ineq: supremum norm bound}. We now consider the diagonal sequence $\{y^n_n\}_{n \in \N}$, then
    \[
        \Vert z - y^n_n \Vert_{\infty} \le \Vert z - z^n \Vert_{\infty} + \Vert z^n - y^n_n \Vert_{\infty} \le \Vert z - z^n \Vert_{\infty} + \Vert z^n \Vert_{C^{\alpha}} 2^{-n\alpha} + C_y 2^{-n\alpha}.
    \]
    From Lemma~\ref{lem: Bernstein poly} below, the right-hand side converges to zero as $n \to \infty$. Therefore, for any $\epsilon > 0$, choosing $n$ sufficiently large and setting $\tilde{z} := y^n_n$ yield $\|z-\tilde{z}\|_{\infty} < \epsilon$.
\end{proof}

\medskip

\begin{lemma}   \label{lem: Bernstein poly}
    For $x \in C([0,1])$, define its $n$-th Bernstein polynomial by
    \[
        x_n(t) := \sum_{k=0}^n x\Big(\frac{k}{n}\Big)\binom{n}{k} t^k(1-t)^{n-k}, \qquad t \in [0,1],
    \]
    so that $\|x_n\|_{\infty} \le \|x\|_{\infty}$ and $\|x_n - x\|_{\infty} \to 0$ as $n \to \infty$. Then, for $\alpha \in (0,1)$, we have the following bound for the $\alpha$-H\"older norm of $x_n$:
    \[
        \|x_n\|_{C^{\alpha}} \le (2n+1)\|x\|_{\infty}, \qquad \forall \, n \in \N.
    \]
\end{lemma}

\begin{proof} [Proof of Lemma \ref{lem: Bernstein poly}]
    It is easy to compute the derivative of the Bernstein polynomial
    \[
        x_n'(t) = n \sum_{k=0}^{n-1} \bigg(x\Big(\frac{k+1}{n}\Big) - x\Big(\frac{k}{n}\Big)\bigg) \binom{n-1}{k} t^k (1-t)^{n-1-k}, \qquad t \in [0, 1].
    \]
    Set
    \[
        b_{n-1,k}(t) := \binom{n-1}{k} t^k (1-t)^{n-1-k},
    \]
    then \(b_{n-1,k}(t) \ge 0\) and 
    \(\sum_{k=0}^{n-1} b_{n-1,k}(t)=1\) for all \(t\in[0,1]\), so
    \[
        \big|x_n'(t)\big|
        \le n \sum_{k=0}^{n-1} \bigg|x\Big(\frac{k+1}{n}\Big) - x\Big(\frac{k}{n}\Big)\bigg| b_{n-1,k}(t)
        \le n \max_{0 \le k \le n-1} \bigg|x\Big(\frac{k+1}{n}\Big) - x\Big(\frac{k}{n}\Big)\bigg| \le 2n\|x\|_{\infty}.
    \]
    Therefore, $\|x_n'\|_{\infty} \le 2n\|x\|_{\infty}$. Since we now have
    \[
        \frac{|x_n(t)-x_n(s)|}{|t-s|^{\alpha}} \le \|x_n'\|_{\infty}|t-s|^{1-\alpha} \le \|x_n'\|_{\infty} \le 2n\|x\|_{\infty}, \qquad \forall \, s,t \in [0,1], ~ s\neq t,
    \]
    it follows that
    \[
        \|x_n\|_{C^{\alpha}} \le \|x_n\|_{\infty} + 2n\|x\|_{\infty} \le (2n+1)\|x\|_{\infty}.
    \]
\end{proof}

\bigskip

\subsection{Proof of Theorem~\ref{thm:q-adic-linear-pvar}}
\begin{proof}[Proof of Theorem \ref{thm:q-adic-linear-pvar}]
The proof follows and generalizes the proof of Theorem~\ref{thm: xi implies p-th variation}; we divide the proof into four steps.

\medskip
    
\noindent\textbf{Step 1: Uniform convergence of the series.}
Since $(\xi_m^{(p,q)})$ converges, the sequence $(y_m)_{m\ge0}$ is bounded, so we define
\begin{equation}    \label{def: upper bound for y_m}
    Y := \sup_m |y_m| < \infty.
    \end{equation}
    Then
    \[
        c_m=q^{m(\frac{1}{2}-\frac{1}{p})}y_m\le Y q^{m(\frac{1}{2}-\frac{1}{p})}.
    \]
    
    For each fixed $m$ and $t\in[0,1]$, there is at most one index $k\in\{0,\dots,q^m-1\}$ such that some $e_{m,k,\ell}(t)$ is nonzero, namely the unique $k$ such that $t\in I_{m,k}$. Moreover, for every $m,k,\ell$,
    \[
        \|e_{m,k,\ell}\|_\infty \le q^{-\frac{m}{2}}.
    \]
    Hence
    \[
        \sup_{t\in[0,1]} \bigg| \sum_{k=0}^{q^m-1}\sum_{\ell=1}^{q-1}\theta_{m,k,\ell}e_{m,k,\ell}(t) \bigg|
        \le \sum_{\ell=1}^{q-1}|a_\ell|\,c_m\,q^{-\frac m2} \le Y\Big(\sum_{\ell=1}^{q-1}|a_\ell|\Big) q^{-\frac mp}.
    \]
    Since $\sum_{m\ge0}q^{-\frac mp}<\infty$, the Weierstrass $M$-test implies that \eqref{eq:q-adic-series} converges uniformly on $[0,1]$ to a continuous function $x$.
    
    \medskip
    
    \noindent\textbf{Step 2: A digit representation of the normalized increments.}
    
    For $n\ge1$ and $k=0,\dots,q^n-1$, write the base-$q$ expansion
    \[
        k=d_1(k)+d_2(k)q+\cdots+d_n(k)q^{n-1}, \qquad d_j(k) \in \{0,\dots,q-1\}.
    \]
    Let
    \[
        \Delta_k^n x:=x\Big(\frac{k+1}{q^n}\Big)-x\Big(\frac{k}{q^n}\Big).
    \]
    
    Fix $n\ge1$, $k\in\{0,\dots,q^n-1\}$, and $m\in\{0,\dots,n-1\}$. The interval
    \[
        I_{n, k} = \Big[\frac{k}{q^n},\frac{k+1}{q^n}\Big)
    \]
    is contained in a unique parent interval $I_{m,\kappa}$ for some $\kappa$, and inside that parent it lies in the child indexed by $d_{n-m}(k)\in\{0,\dots,q-1\}$. Indeed, if $\kappa=\lfloor k/q^{\,n-m}\rfloor$, then
    \[
        \Big\lfloor \frac{k}{q^{\,n-(m+1)}}\Big\rfloor = q\Big\lfloor \frac{k}{q^{\,n-m}}\Big\rfloor + d_{n-m}(k) = q\kappa+d_{n-m}(k),
    \]
    so the unique child of $I_{m,\kappa}$ containing $I_{n,k}$ is $I_{m+1,\,q\kappa+d_{n-m}(k)}$.
    For the convenience of readers, we provide a toy example on how this digit expansion works in Appendix~\ref{App: toy-ex}.

    By the definition of $\psi_{m,\kappa,\ell}$ in \eqref{def: q-adic Haar}, we have on this level-$n$ interval
    \[
        \psi_{m,\kappa,\ell}(t) = q^{\frac m2}\gamma_{\ell,d_{n-m}(k)},
        \qquad t\in \Big[\frac{k}{q^n},\frac{k+1}{q^n}\Big).
    \]
    Hence
    \[
        e_{m,\kappa,\ell}\Big(\frac{k+1}{q^n}\Big)-e_{m,\kappa,\ell}\Big(\frac{k}{q^n}\Big)
        = q^{-n}q^{\frac m2}\gamma_{\ell,d_{n-m}(k)}.
    \]
    Multiplying by $\theta_{m,\kappa,\ell}=c_m a_\ell$, summing over $\ell$, and recalling the definition \eqref{def: eta_d}, the level-$m$ contribution to $\Delta_k^n x$ is
    \[
        q^{-n}q^{\frac m2} c_m \eta_{d_{n-m}(k)}(a),
    \]
    where the notation $\eta_d$ is defined in \eqref{def: eta_d}.
    Therefore
    \[
        \Delta_k^n x = \sum_{m=0}^{n-1} q^{-n}q^{\frac m2} c_m \eta_{d_{n-m}(k)}(a).
    \]
    Reindexing with $j:=n-m$, and replacing $c_m$ with $y_m$ by \eqref{def: y_m q-adic}, we obtain
    \begin{equation}\label{eq:q-adic-increment}
        q^{\frac np}\Delta_k^n x
        = \sum_{j=1}^n q^{-j(1-\frac 1p)} y_{n-j}\eta_{d_j(k)}(a)
        = \sum_{j=1}^n \rho_q^j y_{n-j}\eta_{d_j(k)}(a).
    \end{equation}
    
    \medskip
    
    \noindent\textbf{Step 3: Linear $p$-th variation at $t=1$ and at $q$-adic rational points.}
    
    Let $(D_j)_{j\ge1}$ be i.i.d.\ random variables uniformly distributed on $\{0,\dots,q-1\}$. Since the digit vectors
    \[
        \big(d_1(k),\dots,d_n(k)\big),\qquad k=0,\dots,q^n-1,
    \]
    exhaust all elements of $\{0,\dots,q-1\}^n$ exactly once, \eqref{eq:q-adic-increment} yields

    \begin{equation}   \label{eq:q-adic-Vn-expectation}
        [x]^{(p)}_{\T_q^n}(1)
        = \sum_{k=0}^{q^n-1} |\Delta_k^n x|^p
        = q^{-n}\sum_{k=0}^{q^n-1} \bigg| \sum_{j=1}^n \rho_q^j y_{n-j}\eta_{d_j(k)}(a) \bigg|^p
        = \mathbb E\Bigg[ \bigg| \sum_{j=1}^n \rho_q^j y_{n-j}\eta_{D_j}(a) \bigg|^p \Bigg].
    \end{equation}
    
    Since $(\xi_m^{(p,q)})$ converges, there exists $y\ge0$ such that $y_m\to y$. Set
    \[
        Z:=\sum_{j=1}^{\infty}\rho_q^j\eta_{D_j}(a).
    \]
    Because $\rho_q\in(0,1)$ from \eqref{def: rho_q}, and the random variables $\eta_{D_j}(a)$ take values in the finite set $\{\eta_0(a),\dots,\eta_{q-1}(a)\}$, the series for $Z$ converges almost surely and in $L^p$. Moreover, we claim the following $L^p$-convergence:
    \begin{equation}    \label{L^p convergence}
        \sum_{j=1}^n \rho_q^j y_{n-j}\eta_{D_j}(a)\xrightarrow[n\to\infty]{L^p} yZ.
    \end{equation}
    Since $y_m\to y$, recall $Y:=\sup_m |y_m|<\infty$ from \eqref{def: upper bound for y_m}. Also, since $\{\eta_0(a),\dots,\eta_{q-1}(a)\}$ is finite, we may set
    \begin{equation}    \label{def: M_a}
        M_a:=\max_{0\le d\le q-1} |\eta_d(a)|<\infty.
    \end{equation}
    Then, for every $n$,
    \[
        \Big|\sum_{j=1}^n \rho_q^j y_{n-j}\eta_{D_j}(a)\Big| \le YM_a \sum_{j=1}^\infty \rho_q^j <\infty \qquad \text{a.s.}
    \]
    Moreover, for each fixed $j\ge1$, we have $y_{n-j}\to y$ as $n\to\infty$. Hence,
    \[
        \sum_{j=1}^n \rho_q^j y_{n-j}\eta_{D_j}(a) \xlongrightarrow{n \to \infty} \sum_{j=1}^\infty \rho_q^j y\,\eta_{D_j}(a)=yZ \qquad \text{a.s.},
    \]
    and the above uniform bound allows us to apply dominated convergence in $L^p$, which proves \eqref{L^p convergence}.
    
    Therefore, from \eqref{eq:q-adic-Vn-expectation},
    \begin{equation}\label{eq:q-adic-Vn-limit}
        [x]^{(p)}_{\T_q^n}(1)\xrightarrow{n\to\infty}
        \mathbb E[|yZ|^p] 
        = \Big(\lim_{m\to\infty}\xi_m^{(p,q)}\Big)\, \mathbb E[|Z|^p]
        = C_{p,q,a}\,\lim_{m\to\infty}\xi_m^{(p,q)}.
    \end{equation}
    Here, since $a\neq 0$, the vector $(\eta_0(a),\dots,\eta_{q-1}(a))$ is not identically zero, hence $P(|Z|>0)>0$, and therefore
    \[
        C_{p,q,a} := \mathbb E[|Z|^p] > 0.
    \]
    
    We next extend this to $q$-adic rational points. Fix $m\ge1$ and
    \[
        t=\frac{\ell}{q^m},\qquad \ell\in\{0,\dots,q^m\}.
    \]
    For each $n\ge m$, decompose $\{0,\dots,q^n-1\}$ into $q^m$ consecutive blocks of length $q^{n-m}$:
    \[
        B_{n,m}(r) := \{rq^{n-m},\dots,(r+1)q^{n-m}-1\},
        \qquad r=0, \dots ,q^m-1.
    \]
    Define
    \[
        V_{n,m}(r) := \sum_{k\in B_{n,m}(r)} |\Delta_k^n x|^p.
    \]
    Then
    \[
        [x]^{(p)}_{\T_q^n}(1) = \sum_{r=0}^{q^m-1}V_{n,m}(r), \qquad
        [x]^{(p)}_{\T_q^n}(t) = \sum_{r=0}^{\ell-1}V_{n,m}(r).
    \]
    
    We claim that, for each fixed $m$,
    \begin{equation}    \label{eq:q-adic-block-equipartition}
        \max_{0\le r\le q^m-1} \Big| V_{n,m}(r)-q^{-m}[x]^{(p)}_{\T_q^n}(1) \Big| \xrightarrow{n\to\infty} 0.
    \end{equation}
    Indeed, for $k\in B_{n,m}(r)$, split \eqref{eq:q-adic-increment} into
    \[
        U_{n,m}(k):=\sum_{j=1}^{n-m}\rho_q^j y_{n-j}\eta_{d_j(k)}(a),
        \qquad
        T_{n,m}(r):=\sum_{j=n-m+1}^{n}\rho_q^j y_{n-j}\eta_{d_j(rq^{n-m})}(a).
    \]
    Then for $k\in B_{n,m}(r)$,
    \[
        q^{n/p}\Delta_k^n x=U_{n,m}(k)+T_{n,m}(r),
    \]
    because the digits $d_j(k)$ with $j>n-m$ are constant on the block. Recalling \eqref{def: upper bound for y_m} and \eqref{def: M_a},
    \[
        |U_{n,m}(k)| \le YM_a\sum_{j=1}^{\infty}\rho_q^j =: M,
        \qquad |T_{n,m}(r)| \le YM_a\sum_{j=n-m+1}^{\infty}\rho_q^j \le M,
    \]
    uniformly in $n,m,k$, hence we have $U_{n,m}(k) + T_{n,m}(r) \in [-2M,2M]$. Since $p>1$, the function $u\mapsto |u|^p$ is Lipschitz on the compact interval $[-2M,2M]$, i.e., there exists a constant $L=L(p,M)>0$ such that
    \[
        \big||u+v|^p-|u|^p\big|\le L|v|, \qquad |u|\le M,\ |v|\le M.
    \]
    Therefore,
    \begin{equation*}
        \Big|V_{n,m}(r)-q^{-n}\sum_{k\in B_{n,m}(r)}|U_{n,m}(k)|^p\Big|
        \le L q^{-n}\sum_{k\in B_{n,m}(r)}|T_{n,m}(r)|
        \le L q^{-m}\max_{0\le r\le q^m-1}|T_{n,m}(r)|.
    \end{equation*}
    
    On the other hand, as $k$ ranges over $B_{n,m}(r)$, the vectors $\big(d_1(k),\dots,d_{n-m}(k)\big)$ exhaust $\{0,\dots,q-1\}^{n-m}$ exactly once, independently of $r$. Hence,
    \[
        A_{n,m}:= q^{-n}\sum_{k\in B_{n,m}(r)}|U_{n,m}(k)|^p
    \]
    is independent of $r$. Since the tail sum $T_{n,m}(r)$ tends to $0$ uniformly in $r$ as $n \to \infty$, we have
    \[
        \max_{0\le r\le q^m-1}|V_{n,m}(r)-A_{n,m}|
        \xrightarrow{n\to\infty} 0.
    \]
    Since the convergence above is uniform in $r$ and the number of blocks is $q^m$, we obtain
    \[
        [x]^{(p)}_{\T_q^n}(1) = \sum_{r=0}^{q^m-1}V_{n,m}(r) = q^m A_{n,m}+o_{n\to\infty}(1).
    \]
    Equivalently,
    \[
    A_{n,m}=q^{-m}[x]^{(p)}_{\T_q^n}(1)+o_{n\to\infty}(1).
    \]
    
    Consequently, for $t=\ell q^{-m}$,
    \[
        [x]^{(p)}_{\T_q^n}(t)
        =
        \sum_{r=0}^{\ell-1}V_{n,m}(r)
        =
        \frac{\ell}{q^m}[x]^{(p)}_{\T_q^n}(1)+o_{n\to\infty}(1).
    \]
    Letting $n\to\infty$ and using \eqref{eq:q-adic-Vn-limit}, we obtain
    \begin{equation}\label{eq:q-adic-rational-limit}
        \lim_{n\to\infty}[x]^{(p)}_{\T_q^n}(t)
        =
        t\,C_{p,q,a}\,\lim_{m\to\infty}\xi_m^{(p,q)}
        \qquad\text{for every } t \in \bigcup_{m\ge0}\T_q^m.
    \end{equation}
    
    \medskip
    
    \noindent\textbf{Step 4: Extension to all $t\in[0,1]$.}
    
    Let $t\in[0,1]$. For each $m\ge1$, define
    \[
        t_m^-:=\frac{\lfloor q^m t\rfloor}{q^m},
        \qquad
        t_m^+:=\frac{\lfloor q^m t\rfloor+1}{q^m}.
    \]
    Then $t_m^-\le t\le t_m^+$ and $t_m^+,t_m^- \rightarrow t$. Since $[x]^{(p)}_{\T_q^n}(\cdot)$ is non-decreasing for every $n$, we have
    \[
        [x]^{(p)}_{\T_q^n}(t_m^-)
        \le
        [x]^{(p)}_{\T_q^n}(t)
        \le
        [x]^{(p)}_{\T_q^n}(t_m^+).
    \]
    By \eqref{eq:q-adic-rational-limit},
    \[
        \lim_{n\to\infty}[x]^{(p)}_{\T_q^n}(t_m^\pm)
        =
        t_m^\pm\,C_{p,q,a}\,\lim_{r\to\infty}\xi_r^{(p,q)}.
    \]
    Taking $\liminf$ and $\limsup$ in $n$ and then letting $m\to\infty$, we obtain
    \[
        \lim_{n\to\infty}[x]^{(p)}_{\T_q^n}(t)
        =
        t\,C_{p,q,a}\,\lim_{r\to\infty}\xi_r^{(p,q)}.
    \]
    This proves the theorem.
\end{proof}

\bigskip

\appendix
\section{An example on Step 2 of the proof of Theorem~\ref{thm:q-adic-linear-pvar}}\label{App: toy-ex}

To provide a clearer explanation of the sentence --- ``The interval $I_{n, k}$ is contained in a unique parent interval $I_{m,\kappa}$ for some $\kappa$, and inside that parent it lies in the child indexed by $d_{n-m}(k)\in\{0,\dots,q-1\}$." --- consider the following example: 

Take $q=3$, $n=5$, and $k=71$. Then ternary expansion of $71$ is \(02122_{(3)}\), i.e.,
 \[
    71 = 2 + 2\cdot 3 + 1\cdot 3^2 + 2\cdot 3^3 + 0\cdot 3^4,
 \]
so following the notation in Step 2 of the proof, 
 \[
    d_1(k)=2, \qquad d_2(k)=2, \qquad d_3(k)=1, \qquad d_4(k)=2, \qquad d_5(k)=0,
 \]
 The level-5 interval is
 \[
    I_{5,71} =\Big[\frac{71}{3^5},\frac{72}{3^5}\Big) =\Big[\frac{71}{243},\frac{72}{243}\Big).
 \]
One now has the following: 

\medskip

If $m=4$, then \(I_{5,71}\subset I_{4,23} = [\frac{23}{81}, \frac{24}{81}) = [\frac{69}{243}, \frac{72}{243}) \). This level-4 parent has 3 children, $I_{5, 69}, I_{5, 70}, I_{5, 71}$, and $I_{5, 71}$ is the child with index \(2=d_{5-4}(k)=d_1(k)\).

\medskip
  
If $m=3$, then \(I_{5,71}\subset I_{3,7} = [\frac{7}{27}, \frac{8}{27}) = [\frac{63}{243}, \frac{72}{243})\). This level-3 parent has 3 children, $I_{4, 21}, I_{4, 22}, I_{4, 23}$, and $I_{5,71} \subset I_{4, 23}$ with index \(2=d_{5-3}(k)=d_2(k)\).

\medskip
     
If $m=2$, then \(I_{5,71}\subset I_{2,2} = [\frac{2}{9}, \frac{3}{9}) = [\frac{54}{243}, \frac{81}{243})\). This level-2 parent has 3 children, $I_{3, 6}, I_{3, 7}, I_{3, 8}$, and $I_{5,71} \subset I_{3, 7}$ with index \(1=d_{5-2}(k)=d_3(k)\).

\medskip

Similar relationships hold for $m = 1, 0$. This illustrates why, in general, the child index at level $m$ is \(d_{n-m}(k)\). 

\bigskip

\bigskip

\noindent \textbf{Funding}

\medskip

\noindent Donghan Kim was supported by the National Research Foundation of Korea (NRF) grant funded by the Korea government (MSIT) RS-2025-00513609.

\bigskip

\renewcommand{\bibname}{References}
\bibliography{pathwise1}
\bibliographystyle{apalike}

\end{document}